\documentclass[preprint,12pt]{elsarticle}
\usepackage[left=2.5cm, right=2.5cm,top=2.5cm,bottom=2.5cm]{geometry}

\usepackage{amsmath, amssymb}
\usepackage{amsthm}
\usepackage{physics}
\usepackage[hidelinks]{hyperref}
\usepackage{mathtools}
\usepackage{xcolor}
\newtheorem{theorem}{Theorem}
\newtheorem*{thm*}{Theorem}
\newtheorem{proposition}{Proposition}[section]
\newtheorem{lemma}{Lemma}[subsection]

\theoremstyle{definition}

\usepackage{apptools}
\usepackage{bm}

\usepackage{appendix}

\AtAppendix{\counterwithin{lemma}{section}}

\newtheoremstyle{remark}%
  {}
  {}
  {}
  {}
  {}
  {.}
  {.5em}
  {}
\theoremstyle{remark}

\numberwithin{equation}{section}
\newcommand{\R}{\mathbb{R}}  

\newcommand{\inner}[2]{\langle {#1}, {#2} \rangle} 
\newcommand{\arXiv}[1]{\texttt{arxiv:#1}}

\journal{arXiv.org}

\begin{document}

\begin{frontmatter}



\title{Enhanced Dissipation via time-modulated velocity fields}


\author[Surrey]{Johannes Benthaus}

\author[Surrey]{Camilla Nobili}

\affiliation[Surrey]{organization={School of Mathematics and Physics, University of Surrey},
            city={Guildford},
            postcode={GU2 7XH}, 
            country={UK}}


\begin{abstract}
Motivated by mixing processes in analytical laboratories, this work investigates enhanced dissipation in non-autonomous flows. We study the evolution of concentrations governed by the advection-diffusion equation, where the velocity field is modelled as the product of a shear flow and a time-dependent modulation function $\xi(t)$. The main objective of this paper is to derive quantitative estimates for the energy decay rates, which are shown to depend sensitively on the properties of $\xi$.

We identify a class of time-dependent functions that are bounded by increasing functions, for which we demonstrate super-enhanced dissipation, characterized by energy decay rates faster than those observed in autonomous cases.

Additionally, we explore the case of velocity fields that may be switched on and off over time. Here, the dissipation rates are comparable to those of autonomous flows.
To illustrate our results, we analyse two prototypical flows of this class: one exhibiting a gradual turn-on and turn-off phase, and another that undergoes a significant acceleration following a slow initial activation phase. 

Both results are achieved through the application of the “hypocoercivity” framework, adapted to an augmented functional with time-dependent weights. These weights are designed to \textit{dynamically} counteract the potential growth of $\xi$, ensuring robust decay estimates.

\end{abstract}

\end{frontmatter}
\section{Introduction}

On the set $\Omega=\mathbb{T}^2\times [0,\infty)$, let
the concentration $\Theta=\Theta(x,y,t):\Omega\rightarrow \R$, and velocity field 
$\bm{u}=\bm{u}(x,y,t):\Omega\rightarrow \R^2$ solve  the Cauchy problem for the advection-diffusion equation
\begin{equation}\label{adv}
	\begin{array}{rrl}
		\partial_t \Theta +\bm{u}\cdot\nabla  \Theta-\nu \Delta \Theta&=&0\\
		\nabla\cdot \bm{u}&=&0\\
		\Theta(x,y,0)&=&\Theta_0(x,y)\,,
	\end{array}
\end{equation}
where $\nu\ll 1$ is the molecular diffusivity and $\Theta_0$ is a mean-free initial datum.
Further, we consider the special class of time-dependent shear flows of the type
\begin{equation}
	\bm{u}(x,y,t)=(v(y,t), 0)^T,
\end{equation}
for which the advection-diffusion equation reduces to
\begin{equation}
	\label{fourier_adv}
	\partial_t \hat{\Theta} +v(y,t) ik \hat{\Theta}=\nu (-k^2+\partial_y^2) \hat{\Theta}\,.
\end{equation}
where we conveniently Fourier transformed in the $x$-variable.
In this work, we are interested in demonstrating enhanced dissipation under specific assumptions on the time-dependent velocity fields. Enhanced dissipation is a physical phenomenon intrinsically connected with mixing \cite{zelati2023mixing}. Mixing generates gradients, i.e. high frequencies, and its effect for inviscid fluids has been widely studied (see the excellent survey \cite{thiffeault2012using} and references therein). 
For viscous fluids, diffusion acts to smooth out irregularities, effectively damping high-frequency components. The interplay between advection and diffusion leads to intricate physical phenomena \cite{Miles_2018, pierrehumbert1994tracer, thiffeault2004strange}, including the phenomenon known as \textit{enhanced dissipation}. Although physically well understood, its mathematical characterization has only recently been explored in specific settings due to its inherent complexity. In their seminal paper \cite{Constantin2008}, Constantin et al. characterize flows that induce enhanced dissipation, in terms of the spectral properties of the dynamical system associated with it. In \cite{Bedrossian2017} the authors quantify the enhanced dissipation effect due to shear by employing the hypocoercivity method introduced by Villani \cite{villani}. In particular, assumptions on the form of the shear are made, and the rates depend significantly on the maximal order of degeneracy of the flow's critical points. As already mentioned, enhanced dissipation and mixing are connected phenomena: the authors in \cite{Elgindi2019} establish a precise connection between quantitative mixing rates in terms of decay of negative Sobolev norms and enhanced dissipation time-scales. Mixing with diffusion is also at the base of the Batchelor scale conjecture, as stated by Charlie Doering in \cite{Miles_2018}. This conjecture asserts that the filamentation length, defined as the ratio of the $H^{-1}$-norm to the $L^2$-norm of the scalar concentration, converges to a constant known as the \textit{Batchelor scale} in the long-time limit. This provides an equivalent formulation, from a PDE perspective, of the theory of strange eigenmodes \cite{pierrehumbert1994tracer}, which is rooted in a dynamical systems approach. In \cite{nobili2022lower} the authors study the long time asymptotics of the filamentation length in the whole space, where they could make use of Fourier splitting technique. 
	Recent progress on the Batchelor-scale conjecture is presented in \cite{Blumenthal2023}, where the authors demonstrate that, under general assumptions, the top Lyapunov exponent associated with a dissipative linear evolution equation is independent of the choice of norm.
	In \cite{hairer2024}, the authors investigate the top Lyapunov exponent for the linear advection–diffusion equation and the linearised Navier–Stokes equations in vorticity form, with the velocity being solution of the stochastic Navier–Stokes equations driven by non-degenerate white-in-time noise with power-law correlations. They establish a lower bound on the exponent by a negative power of the diffusivity, providing the first such bound on the Batchelor scale.
\\

For autonomous shear flows, enhanced dissipation is well-established in the literature, and we therefore point towards the survey \cite{mazzucato2025} for reference. In \cite{cotizelatigallay} the authors, inspired by \cite{Bedrossian2017}, establish enhanced dissipation in the infinite channel \( \mathbb{R} \times [0,1] \), assuming that \( \Theta \) satisfies Neumann boundary conditions and that the velocity profile \( v(y) \) is either monotonic or has only finitely many simple critical points. These critical points, denoted \( \{y_i\}_{i=1}^N \), satisfy
\begin{equation}\label{simple-critical}
	v''(y_j) \neq 0 \quad \text{for } j = 1, \dots, N.
\end{equation}
Under these conditions, they demonstrate that the decay rate is governed by the maximal order of the critical points. We denote this order by $m$, setting $m=1$ in the monotonic case and $ m = 2 $ in the case of simple critical points. They then obtain the decay estimate
\begin{equation}
	\label{classical_ed}
	\|\hat{\Theta}(t)\|_{L^2_y} \leq C_1 \exp(-C_2 \lambda_{\nu,k} t - \nu k^2 t)\|\hat{\Theta}_0\|_{L^2_y} \qquad \mbox{ for all }t\geq 0
\end{equation}
with
\[
\lambda_{\nu,k} =
\begin{cases}
	\nu^{\frac{m}{m+2}} |k|^{\frac{2}{m+2}} & \text{if } 0 < \nu < |k|, \\[6pt]
	\frac{k^2}{\nu} & \text{if } 0 < |k| < \nu.
\end{cases}
\]
Notably, the authors demonstrate that this result can be equivalently derived using both a spectral approach (à la Wei \cite{wei2021diffusion}) and the hypocoercivity method. The above rate is sharp for time-independent flows, as proven in \cite{Zelati2021}. Hence, the above decay estimate for the case of $m=2$ will be our main point of reference when referring to the autonomous setting.

In $2D$, a generalization of this result to the case where $v$ admits a finite number of possibly degenerate critical points, as well as to geometries like the torus $\mathbb{T}^2$ and the periodic channel $\mathbb{T}\times[0,1]$, can be found in the aforementioned seminal work \cite{Bedrossian2017}.  We also want to highlight the work of \cite{Brue2024}, where the authors prove enhanced dissipation (and mixing) for Hamiltonian flows. In the specific case of cellular flows the authors also establish lower bounds on the mixing rate. 

\medskip
In contrast, the case of a time-dependent velocity field, although physically relevant, has been  explored far less in the mathematics literature. Recently, enhanced dissipation-type results have been established in the time-periodic setting \cite{elgindi2023}, and in the stochastic context \cite{Bedrossian2021,cooperman2025,Gess2025,Seis2025}.
 \\Our interest in this scenario is motivated by practical applications. In analytical laboratories, cell disruption is commonly achieved using vortex mixers, which generate violent, high-speed vortexing action. Starting from a state of rest, these mixers rapidly accelerate to very high speeds; this allows for fine adjustments in the mixing intensity, ranging from gentle agitation to vigorous vortexing. Just as crucial is the controlled deceleration of the mixer back to zero, ensuring that sensitive structures are not damaged during the slowdown phase. In chemical and analytical laboratories, a wide variety of stirring devices are employed, each carefully designed to achieve optimal mixing by generating a strongly time-dependent stirring field.

Motivated by these considerations, in this paper we aim at quantifying the effect of time dependent velocity fields in enhanced dissipation estimates. One major challenge in the mathematical analysis of such flows stems from the limitations of the hypocoercivity method when applied to general time-dependent velocity fields, as we will elaborate in the following discussion.\\ To this end, we consider velocity fields for which the $v(y,t)$ part separates into a space-time product of the form
\begin{equation}\label{vector-field}
	v(y,t) = \xi(t) v(y), \quad \text{with } \xi(t)\geq 0,
\end{equation}
reducing equation \eqref{fourier_adv} to 
\begin{equation}
	\label{product_adv}
	\partial_t \hat\Theta + \xi(t) v(y) ik\hat\Theta= \nu (-k^2+\partial_y^2) \hat\Theta.
\end{equation}
This class of velocity fields for the case of $\xi(t)$ being a piecewise constant random function was already considered in \cite{Vanneste2006}. In order to study the decay rate of $\|\Theta(t)\|_{L^2_y}$ in the regime where advection dominates diffusion, i.e. small viscosities, the author performs a boundary-layer analysis effectively reducing the problem to the study of a pair of coupled stochastic differential equations independent of diffusivity. In this paper, we will consider only deterministic functions $\xi(t)$ in $L^2([0,T])$ and shear velocities $v\in H^1_y(\mathbb{T})$, for which well posedness holds \cite{Evans2010} (see also \cite{crippa2024}).

In order to illustrate how the time-dependent function $\xi(t)$ affects the decay rates of the scalar concentration, we first consider the inviscid case $\nu=0$. Applying the method of stationary phases \cite{Stein}, it is easy to show the  upper bound
\begin{align*}
	\norm{\hat{\Theta}(k,t)}_{H_{y}^{-1}}\leq\min\left\{\frac{C}{(\abs{k}\Xi(t))^{1/2}},1\right\}\norm{\hat{\Theta}_0(k)}_{H_{y}^1}\,,
\end{align*}
where $\Xi(t)=\int_0^t\xi(\tau)\, d\tau\,$
(see proof in Appendix \ref{A:mix}). This coincides with the mixing estimate in \cite{Bedrossian2017} when $\xi=1$.

We observe that, if $\xi(t)$ increases sufficiently quickly (for instance, if $\xi(t)\gtrsim\alpha t^{\alpha-1}$ for some $\alpha>1$ and all $t\leq T$), the upper bound indicates that the mixing is enhanced over long times compared to the autonomous scenario. If $\xi(t)$ is bounded from below and above by constants, then the result in \cite{Elgindi2019} immediately provides an enhanced dissipation estimate for the advection diffusion equation (see discussion in Appendix \ref{A:mix}).

Recently, in some particular settings, time dependent velocity fields have been investigated in relation to their dissipation enhancing properties.
In \cite{beck}, the authors investigate the two-dimensional linearized Navier-Stokes equations around the Kolmogorov flow, defined as 
\[
V_{\rm{NS}} = \big(-a e^{-\nu t} \cos y, 0\big)^T.
\]
By neglecting the non-local component of the linearized operator 
\[
\mathcal{L}_{\nu} = -\nu \Delta - a e^{-\nu t} \cos y \partial_x (1 + \Delta^{-1}),
\]
they focus on the vorticity equation
\[
\partial_t \omega - \nu \Delta \omega - a e^{-\nu t} \cos y \partial_x \omega = 0.
\]
Using the hypocoercivity method, they establish the enhanced dissipation estimate
\begin{equation}\label{Wei-est}
	\|\omega(t)\|_X \lesssim e^{-c\sqrt{\nu}t} \|\omega_0\|_X,
\end{equation}
where $X\subset L^2$ is a Banach space (see \cite[Theorem 3.2]{beck}).

In \cite{Wei2019}, the authors revisited the non-local component of the linearized operator and established 
\eqref{Wei-est} with $X=L^2(\mathbb{T}^2_{\delta} )$ where \(\mathbb{T}^2_{\delta} = \{(x,y) : 0 < x < 2\pi \delta, \, 0 < y < 2\pi\}\) with \(\delta < 1\).

Velocities of the form $e^{-\nu t}v(y)$, where $v \in \mathcal{C}^2$, also belong to the class analysed in \cite{Coble2024}. There the authors obtain an extension of \cite{Bedrossian2017} in the sense that they recover the classical result of \eqref{classical_ed} $(m=1,2)$ for certain time-dependent flows. We also note the insightful work in \cite{seis2023bounds}, where the author, using optimal transport techniques, provides upper bounds on the exponential rates of enhanced dissipation for a broad class of velocity fields without imposing specific structural assumptions, relying only on regularity conditions such as
$
\int_{0}^{t} \|\nabla \mathbf{u}\|_{L^p} \lesssim 1 + t^\alpha.
$

\medskip

A central and novel aspect of our work lies in the derivation of time-sensitive estimates that precisely capture the evolution of the non-autonomous flow, avoiding bounding $\xi(t)$ by a constant in key steps. In order to achieve this, we embed the flow's dynamics directly into the hypocoercivity framework. For the targeted class of flows, the proof of our first result leverages time-dependent weights that are dynamically linked to $\xi(t)$, while in our second result we introduce it directly into the functional.

Our first result is an enhanced dissipation estimate for a specific class of velocity fields for which the time-dependent part satisfies polynomial lower and upper bounds.
\begin{theorem}\label{T:hyp_weights_main}
	Let $\nu\ll 1$ and $\hat{\Theta}$ be a solution of \eqref{product_adv} with $\hat{\Theta}_0$ in $L^2(\mathbb{T})$.
	Suppose that $v:\mathbb{T}\rightarrow \mathbb{R}$ is a $\mathcal{C}^2$ function admitting only a finite number of simple critical points (i.e. satisfying \eqref{simple-critical}). Define the weights $w_{\nu}(t):=(1+\nu^s t)^{-1}$ for some $s\in [0,\infty)$.  
	Let  $\beta\in [\frac{\nu}{|k|},1]$ and assume that $\xi(t)$ satisfies
	
	\begin{equation}\label{LUb}
		L(t)\leq \xi(t)\leq U(t) \qquad \forall t\in [0,T]\,,
	\end{equation}
	where 
	\begin{equation*}
		\begin{array}{rl}
			L(t) &= 
			\begin{cases}
				C^2\beta^2 & \text{for } t \leq t^{*}_{\nu, k} 
				\\
				\frac{\nu}{|k|\beta} w_{\nu}(t)^{-2} & \text{for } t \geq t^{*}_{\nu, k}
			\end{cases} 
			\\[8pt]
			& \quad \text{where} \quad t^{\ast}_{\nu,k}=\frac{1}{\nu^s}\left(\frac{|k|^{\frac 12}C\beta^{\frac 32}}{\nu^{\frac 12}}-1\right), 
			\\[15pt]
			U(t) &= w_{\nu}(t)^{-\ell} 
			\quad \text{for some } \ell \in [2,4],
		\end{array}
	\end{equation*}
	where $C$ is some positive constant. 
	Then the $L^2$ decay estimate 
	\begin{multline}
		\|\hat{\Theta}(t)\|_{L^2_y}^{2} \leq C_{\rm{ed}}\qty(1+\qty(\frac{\abs{k}}{\nu})^{1/2}) \\ \times\exp\left(-\frac{1}{4}(\beta\nu\abs{k})^{\frac 12} \int_0^t \xi(\tau)^{\frac 12}w_{\nu}(\tau)\, d\tau-2\nu k^2 t\right)  \norm{\hat{\Theta}_0}_{L^2_y}^{2} 
	\end{multline}
	holds for any $t\leq T$ where the time $T$ is arbitrary but finite and $C_{\rm{ed}}$ is some constant independent of $\nu,k$.
\end{theorem}
\begin{figure}[ht]
	\label{fig1}
	\centering
	\includegraphics[width=5.0truein]{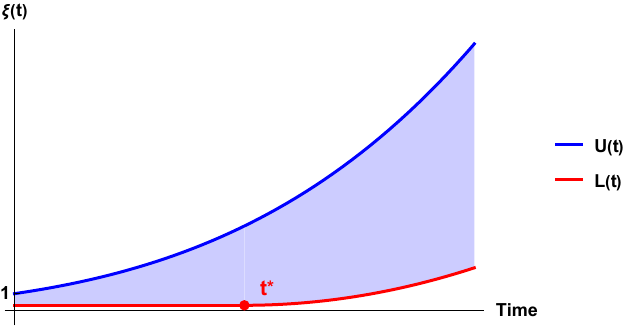}
	\caption{Illustrative sketch of the upper and lower bounds for $\xi(t)$ in Theorem \ref{T:hyp_weights_main}. The shaded area is the admissible region for $\xi(t)$.}
\end{figure}

We observe that when $w_{\nu}=1$ (formally corresponding to $s\rightarrow\infty$), for $\xi(t)$ satisfying 
\begin{equation}\label{class0}
	C^2\beta^2\leq \xi(t)\leq 1 \qquad \forall t\in [0,T]\,,
\end{equation}
the result in Theorem \ref{T:hyp_weights_main} reads
\begin{multline}\label{class0-bound}
	\|\hat{\Theta}(t)\|_{L^2_y}^{2}  \leq C_{\rm{ed}}\qty(1+\qty(\frac{\abs{k}}{\nu})^{1/2})\\ \times \exp\left(-\frac{1}{4}(\beta\nu\abs{k})^{\frac 12} \int_0^t \xi(\tau)^{\frac 12}\, d\tau-2\nu k^2 t\right)  \norm{\hat{\Theta}_0}_{L^2_y}^{2} .
\end{multline}
Note that the prefactor  $1+\qty(\frac{\abs{k}}{\nu})^{1/2}$ in the above decay estimates represents a logarithmic correction and is of purely technical origin (as observed in \cite{Bedrossian2017,cotizelati2019}), but might be removed by the employment of time-dependent weights as in \cite{Wei2019}.

The bound in \eqref{class0-bound} holds for $\xi=1$, corresponding to the class of time-independent velocities considered in \cite{cotizelatigallay, Bedrossian2017}. We note that restricting the analysis to an arbitrary finite time $T$ is necessary to prevent the bounds for $\xi(t)$ in \eqref{LUb} from blowing up. This restriction ensures that $v(t,y) \in L^2([0,T])$, thereby recovering well-posedness from the results in \cite{crippa2024}, as discussed earlier. Practically, this restriction is insignificant since the exponential decay allows the remaining norm after the validity of the estimates to become arbitrarily small by choosing an appropriate $T$.
	
	A commonly studied example is the function $\xi(t)=e^{-\nu t}$ (see \cite{beck,Coble2024,Wei2019}). Selecting the maximal time to match the purely diffusive timescale $T=\frac{1}{\nu}$, at which enhanced dissipation effects have manifested and the remaining norm is minimal (the exponent scales as $\mathcal{O}(\nu^{-1/2})$), ensures that $\xi(t)=e^{-\nu t}$ satisfies the bounds in \eqref{class0}, with the lower bound provided by the constant $e^{-1}$.
 Bounded oscillatory functions such as $\xi(t)=\frac{1}{4}\cos(t)+\frac 12$ are also admissible in the class \eqref{class0}. For all these functions, bounded from above and below by constants as in \eqref{class0}, our estimate indicates enhanced dissipation with a rate comparable in scaling to $e^{- c(|k|\nu)^{\frac 12} t}$. 
Our result in Theorem \ref{T:hyp_weights_main} however goes beyond this class, allowing growth of the velocity in time, yielding a rate that is actually faster than $e^{- c(|k|\nu)^{\frac 12} t}$. We will illustrate such a \textit{super} enhanced dissipation phenomena with an example. Fix $s=\frac 14$ and consider
\begin{equation}\label{ex}
	\xi(t)=(1+\nu^{1/4} t)^4\,,
\end{equation}
corresponding to the fastest growth allowed from the upper bound in \eqref{LUb}.
The integral in the decay rate can be computed explicitly as
\begin{align*}
	\int_0^t \xi(\tau)^{\frac 12}w_{\nu}(\tau)\, d\tau=  \int_0^t (1+\nu^{1/4}\tau)\, d\tau=t+\frac{\nu^{1/4}}{2}t^2.
\end{align*}
yielding the exponential decay 
$e^{-\bar{c}(\nu\abs{k})^{1/2}\qty(\frac{\nu^{1/4}}{2}t^2+t)}$, which is clearly faster than $e^{-\bar{c}(\nu\abs{k})^{1/2} t)}$. To see this, recall that, on the one hand, at the enhanced dissipation time $t_{\rm{ed}}=\frac{1}{\nu^{1/2}}$ the time-independent velocity field considered in \cite{cotizelatigallay} (and all the velocities of the type $\xi(t)v(y)$ where $\xi(t)$ is the class \eqref{class0})  produces a factor $\sim e^{-\abs{k}^{1/2}}$. On the other hand, for the velocity field defined by \eqref{ex}, at time $t_{\rm{ed}}=\frac{1}{\nu^{1/2}}$ we have
\begin{align*}
	\int_0^{t_{ed}} (1+\nu^{1/4}\tau)\, d\tau=\frac{1}{2}\frac{1}{\nu^{3/4}}+\frac{1}{\nu^{1/2}} \Rightarrow   \|\hat{\Theta}(t_{\rm{ed}})\|_{L^2_y}^{2}   \lesssim e^{-\frac{\abs{k}^{1/2}}{\nu^{1/4}}}.
\end{align*}
implying that, by this time, the energy is almost entirely dissipated, revealing a notable acceleration of the mixing process of orders of $\nu$ (see the discussion in Appendix \ref{A:mix}). 

The proof of Theorem \ref{T:hyp_weights_main} builds on the hypocoercivity framework developed by Villani \cite{villani}, which has been subsequently adapted to the passive scalar setting via a series of works, including \cite{beck,Bedrossian2017,Gallagher2009}. For simplicity, we assume that $v(y)$ only admits simple critical points. In order to treat the time dependent coefficient in our advection-diffusion equation, i.e. $\xi(t)$, we define a weighted functional, where polynomial weights in time have the function of balancing the possible growth of $\xi(t)$, which otherwise could make the hypocoercivity method (in its more standard formulation) inapplicable.

We observe that, differently from \cite{Wei2019} (and \cite{Coble2024}), our time-dependent weights are used \textit{dynamically}, as one can appreciate from the constraint \eqref{LUb} (see also Figure \ref{fig1}). Lastly, alternative selections of weight functions might also be permissible (see discussion in Section \ref{Section-product}, after \eqref{weight_derivative}).\\

To build intuition about the effect of  $\xi(t)$ on the decay rate, we analyse the energy and enstrophy balances:
\begin{align}
	\frac{1}{2}\dv{}{t}\norm{\hat{\Theta}}^2_{L^2_y} &= -\nu \qty( \norm{\partial_{y} \hat{\Theta}}^2_{L^2_y}+k^2\norm{ \hat{\Theta}}^2_{L^2_y} ), \\
	\frac{1}{2}\dv{}{t} \norm{\partial_y\hat{\Theta}}^2_{L^2_y} &= -\nu\qty( \norm{\partial^2_{y} \hat{\Theta}}^2_{L^2_y}+k^2\norm{ \partial_y\hat{\Theta}}^2_{L^2_y} )- \xi(t)\inner{ik\partial_y v \hat{\Theta}}{\partial_y \hat{\Theta}}.
\end{align}

The first balance shows that the decay of the solution's $L^2$-norm is controlled by the gradient's norm and the viscosity. Using only this balance, decay occurs on a diffusive timescale, as guaranteed by the Poincaré inequality, when applicable. 

The second balance governs the gradient norm's growth. For enhanced dissipation to occur, the gradient's time derivative must be positive for some time. However, by the first balance, it cannot remain positive indefinitely, as gradients eventually decay. This necessitates at least one sign change in the rate of growth, as noted for specific examples in \cite[Fig. 4]{Young82}, \cite[Fig. 3]{Foures2014} and \cite[Fig. 2]{Miles_2018}.  Here, the first term on the right is strictly negative, so any gradient growth must stem from the sign-indefinite term, tied to the transport operator and modulated by $\xi(t)$. Increasing $\xi(t)$ accelerates growth, while decreasing $\xi(t)$ dampens it.

In summary, larger gradients yield faster norm decay via the first balance. Since gradient growth depends on $\xi(t)$, peak gradients can be reached earlier for increasing $\xi(t)$ and later (or lower peaks altogether) for decreasing $\xi(t)$. This heuristic argument is captured precisely in Theorem \ref{T:hyp_weights_main}.

In order to accommodate functions that might not be bounded from below and therefore might switch on and off at certain intervals in time, we prove
\begin{theorem}\label{T:on_off}
	Let $\nu\ll 1$ and $\Theta$ be a solution of \eqref{product_adv} with $\hat{\Theta}_0$ in $L^2(\mathbb{T})$.
	Suppose that $v:\mathbb{T}\rightarrow \mathbb{R}$ is a $\mathcal{C}^2$ function admitting only a finite amount of simple critical points (i.e. satisfying \eqref{simple-critical}). Suppose that $\xi(t)$ satisfies 
	\begin{equation}\label{class2}
		\xi(t) \leq 1, \quad \dv{}{t}\xi(t) \leq {C_{\xi}}(\nu\abs{k})^{1/2},
	\end{equation}
	for some small constant $C_{\xi}$.
	Then for each $\nu$ and $k$ there exists a $C_\xi' \geq \frac{\nu}{\abs{k}}$ such that the $L^2$ decay estimate 
	\begin{multline}
		\label{onoff_decay}
		\|\hat{\Theta}(t)\|_{L^2_y}^{2}  \leq C_{\rm{ed}}\qty(1+\qty(\frac{\abs{k}}{\nu})^{1/2}) \\
		\times\exp\left(-C_\xi'(\nu\abs{k})^{\frac 12} \int_0^t \xi(\tau)^{3}\, d\tau-2\nu k^2 t\right)  \norm{\hat{\Theta}_0}_{L^2_y}^{2} 
	\end{multline}
	holds for any $t\leq T$ where the time $T$ is arbitrary but finite and $C_{\rm{ed}},C_\xi'$ are constants independent of $\nu,k$.
\end{theorem}
The value of this theorem is to provide the energy decay estimate for velocity fields that are very small or turned off for some time. In Section \ref{intermittent-flow}, we will discuss in details velocities of the type $\xi(t)v(y)$ where 
$$
\xi_A(t) = 
\begin{cases}
	\nu^{\frac 1 2} t, & t \in \qty[0, \frac{1}{\nu^{\frac 1 2}}], \\
	(1 + \nu^{\frac 14} (t - \frac{1}{\nu^{\frac 1 2}}))^4, & t \in \qty[\frac{1}{\nu^{\frac 1 2}}, \frac{1}{\nu^{\frac 3 4}}].
\end{cases}
\;\text{and}\;
\xi_B(t) =
\begin{cases}
	\nu^{1/2} t & t \in \qty[0, \frac{1}{\nu^{1/2}}], \\
	1,& t \in \qty[\frac{1}{\nu^{\frac 1 2}}, \frac{1}{\nu^{\frac 3 4}}], \\
	\frac{\nu t - 1}{\nu^{\frac 1 4} - 1},& t \in \qty[\frac{1}{\nu^{\frac 3 4}}, \frac{1}{\nu}].
\end{cases}
$$

The flow corresponding to \(\xi_A\) exemplifies a transition from a slow flow amenable to Theorem \ref{T:on_off} to a fast flow satisfying the assumptions of Theorem \ref{T:hyp_weights_main}. We demonstrate that the decay rate in this setting can be derived via a ``gluing'' argument. Meanwhile, the flow corresponding to \(\xi_B\) serves as an illustrative example of a velocity profile that activates initially, remains constant over a specified interval, and then deactivates, reflecting a typical scenario in laboratory mixers.
As we will demonstrate in Section \ref{intermittent-flow}, applying Theorem \ref{T:hyp_weights_main}/Theorem \ref{T:on_off} to these profiles results in a decay rate that differs only slightly from the rate observed in the autonomous case of $\xi(t) = 1$. These differences are primarily reflected in lower-order corrections producing small scaling variations, which emerge only after most of the energy has already dissipated.
Furthermore, for the flow corresponding to $\xi_{B}$, we observe that the exact shape of the turn-off function has little influence on the dissipation rate  after the enhanced dissipation timescale is reached.

To provide another example, in the time interval $[0,\frac{1}{\nu}]$ we consider $\xi(t)=e^{-t}$. While Theorem \ref{T:hyp_weights_main} does not apply, since this function is only bounded from below by the $\nu$-dependent function $e^{-\frac{1}{\nu}}$, it satisfies \eqref{class2} and therefore Theorem \ref{T:on_off} is applicable.

This then yields, as expected, a slowed dissipation rate compared to the autonomous case.

The rest of the paper is devoted to the proof of the main results: Theorem \ref{T:hyp_weights_main} is proven in Section \ref{Section-product}, while Section \ref{S:inter} is devoted to the proof of Theorem \ref{T:on_off} and to the computation of rates for some example flows.
\\
\\
In this manuscript we occasionally use the symbols $\lesssim$, $\gtrsim$ to denote inequalities which hold up to a universal positive constant.

\section{Enhanced dissipation for accelerating velocity fields}\label{Section-product}

We consider velocity fields 
\begin{equation}
	\label{setup_1}
	\bm{u}(x,y,t)=(\xi(t)v(y), 0)^T,
\end{equation}
for which the advection-diffusion equation \eqref{adv} reduces to 
\begin{equation}
	\partial_t \Theta +\xi(t)v(y)\partial_x\Theta-\nu \Delta \Theta=0\,.
\end{equation}
We assume 
\begin{equation*}
	\int_{\mathbb{T}}\Theta_{0}(x,y)\,dx=0 \quad \mbox{ for all } y\in \mathbb{T}
\end{equation*}
implying 
\begin{equation}
	\label{avg_free}
	\int_{\mathbb{T}}\Theta(x,y)\,dx=0 \quad \mbox{ for all } y\in \mathbb{T}
\end{equation}
since these preserved are conserved through the evolution. 
As outlined in the introduction, the structure of the shear flow and the domain's periodicity make it natural to transition to the Fourier transform in $x$, leading to the equation
\begin{equation}
	\label{product_adv_sec}
	\partial_t \hat{\Theta} +\xi(t)v(y) ik \hat{\Theta}=\nu (-k^2+\partial_y^2) \hat{\Theta},
\end{equation}
where $\hat{\Theta}(k, y,t)=\frac{1}{2\pi}\int_{\mathbb{T}}\Theta(x, y,t) e^{-ikx}\,dx$ are the Fourier coefficients in the Fourier series expansion $\Theta(x,y,t)=\sum_{k\neq 0}\hat{\Theta}(y,t) e^{ikx}$. 
By the transformation
\begin{align}
	\label{transformation}\hat{\Theta}(k, y, t)=e^{-\nu k^2 t}\theta(k,y,t), \quad  k\in\mathbb{Z}, \quad y\in \mathbb{T}, \quad t>0\end{align}
equation turns \eqref{product_adv_sec} into the hypoelliptic equation
\begin{equation}\label{Heu:adv-eq}
	\begin{array}{rrl}
		\partial_t \theta +\xi(t)v(y) ik \theta&=&\nu \partial_y^2 \theta\,,\\
		\theta(x,y,0)&=&\theta_0(x,y)\,,
	\end{array}
\end{equation}
which hides the horizontal diffusion term $-\nu k^2 \hat{\Theta}$.\\
In the rest, for two complex valued functions $\phi,\psi$,  we will denote with
\begin{align}
	\label{setup2}
	\inner{\phi}{\psi}=\int_\mathbb{T} \phi\bar{\psi} dy.
\end{align}
the standard inner product in $L^2_y(\mathbb{T})$.

We begin by introducing the augmented functional 
\begin{equation}\label{aug-functional}
	\Phi \coloneqq \frac{1}{2}\bigl[E_0(t) \;+\; \alpha_0 (w_{\nu}(t))^{3} E_1(t) \;+\; 2\beta_0 (w_{\nu}(t))^{2} E_3(t) \;+\; \gamma_0 w_{\nu}(t) k^2 E_4(t)\bigr],
\end{equation}
where 
\begin{equation}\label{Es}
	\begin{aligned}
		& E_0(t) = \|\theta(\cdot,t)\|_{L^2_y}^{2} , \qquad
		E_1(t) = \|\partial_y\theta(\cdot,t)\|_{L^2_y}^{2} , \qquad
		E_2(t) = \|\partial_y^2\theta(\cdot,t)\|_{L^2_y}^{2} , \\[5pt]
		&E_3(t) = \Re\langle ik\partial_y v\theta(\cdot,t), \partial_y \theta(\cdot,t)\rangle, \qquad
		E_4(t) = \|\partial_y v\,\theta(\cdot,t)\|_{L^2_y}^{2} .
	\end{aligned}
\end{equation}
The time-dependent weights, parametrized by $\nu$, are defined as
\begin{equation}\label{weightdefinition}
	w_{\nu}(t) = \frac{1}{1+\nu^s t}, \quad \text{for some fixed } s \in [0,\infty).
\end{equation}
As we will see later in the proof of Theorem \ref{T:hyp_weights_main}, these weights are carefully chosen to counteract the (potentially) increasing values of $\xi(t)$ over time.
They satisfy the relation
\begin{equation}\label{weight_derivative}
	\frac{d}{dt}(w_{\nu}(t)^n) = -n\,\nu^s w_{\nu}(t)^{n+1}, \quad \text{for all } n \in \mathbb{N}.
\end{equation}
Before we move on, we want to briefly discuss the structure of the weights. The crucial property of the weights is their decay, as captured in \eqref{weight_derivative}. The exact polynomial structure on the other hand was chosen to introduce this method via the most simple case, and we expect that there may exist others that work similarly.
The parameter $s$ allows us to control the bounds of Theorem \ref{T:hyp_weights_main} in terms of timescales measured in powers of the viscosity. This is convenient, as enhanced dissipation and diffusion are usually measured in terms of these scales.

The goal is to show that the functional $\Phi$ satisfies an estimate of the type (see Lemma \ref{L:functionaldecayprop})
$$\frac{d}{dt}\Phi(t)\lesssim -(\xi(t) \nu |k|)^{\frac 12}w_{\nu}(t)\Phi(t) \qquad \mbox{ for any} \quad t\leq T\,.$$
This will yield our result upon the application of Grönwall's Inequality. \\

The coefficients $\alpha_0$, $\beta_0$, and $\gamma_0$ must satisfy the following inequality
\begin{equation}\label{condition_param_2}
	\frac{\beta_0^2}{\alpha_0} \leq \frac{\gamma_0}{16},
\end{equation}
and we make the choice
	\begin{align}
		\label{beta_def}
		\alpha_0=(\beta_0 \nu)^{1/2}, \quad \beta_0=\frac{\beta}{\abs{k}}, \quad\gamma_0=16\frac{\beta_0^{3/2}}{\nu^{1/2}}
	\end{align}
	where $\beta \in \qty[\frac{\nu}{\abs{k}},1]$ is some constant to be fixed later.\\
\\
Notice that functional $\Phi$ is coercive in the following sense
\begin{align}
	\label{coercivity}
	\frac{1}{8}\qty(4E_0+3\alpha_0 w_{\nu}^{3}E_1+3\gamma_0 k^2 w_{\nu}E_4)\leq \Phi \leq\frac{1}{8}\qty(4E_0+5\alpha_0 w_{\nu}^{3} E_1+5\gamma_0 k^2 w_{\nu}E_4).
\end{align}
Using relation \eqref{condition_param_2},
this is obtained by a direct application of the Cauchy-Schwarz and Young's inequalities as
\begin{align*}
	2\beta_0w_{\nu}^{2} \lvert E_3 \rvert \leq \frac{\alpha_0}{4}w_{\nu}^{3}E_1+\frac{4\beta_0^2k^2}{\alpha_0 }w_{\nu}E_4\leq \frac{\alpha_0}{4}w_{\nu}^{3}E_1+\frac{\gamma_0k^2}{4}w_{\nu}E_4.
\end{align*}

The crucial ingredient, in order to deduce the evolution of the $L^2$-norm of the solution, is the following quantitative bound on $\Phi$:

\begin{proposition}\label{P:func_decay}
	Let $\nu$ and $k$ such that $\beta\in [\frac{\nu}{|k|},1]$ and assume that $\xi(t)$ satisfies
	\begin{align}
		\label{gamma_bounds}
		L(t)\leq \xi(t)\leq U(t)
	\end{align}
	where 
	\begin{equation*}
		\begin{array}{rl}
			L(t) &= 
			\begin{cases}
				C^2\beta^2 & \text{for } t \leq t^{*}_{\nu, k} 
				\\
				\frac{\nu}{|k|\beta} w_{\nu}(t)^{-2} & \text{for } t \geq t^{*}_{\nu, k}
			\end{cases} 
			\\[8pt]
			& \quad \text{where} \quad t^{\ast}_{\nu,k}=\frac{1}{\nu^s}\left(\frac{|k|^{\frac 12}C\beta^{\frac 32}}{\nu^{\frac 12}}-1\right), 
			\\[10pt]
			U(t) &= w_{\nu}(t)^{-\ell} 
			\\[8pt]
			& \quad \text{for some } \ell \in [2,4]
		\end{array}
	\end{equation*}
	with $s\in [0,\infty)$ and some positive constant $C$. 
	Then the energy functional $\Phi(t)$ satisfies the following estimate
	\begin{equation}
		\Phi(t)\leq \exp\left(-\frac{1}{4}(\beta\nu\abs{k})^{\frac 12} \int_0^t \xi(\tau)^{\frac 12}w_{\nu}(\tau)\, d\tau\right)\Phi(0)
	\end{equation}
	for any $t\leq T$ where the time $T$ is arbitrary but finite. 
\end{proposition}
For the proofs of Proposition \ref{P:func_decay} and Theorem \ref{T:hyp_weights_main} the following energy balances are crucial. They result form standard energy methods and proofs are provided in Lemma \ref{L:energy} in the Appendix. We collect them here as a central reference
	\begin{align}
		\frac{1}{2}\frac{d}{dt}E_0(t) &= -\nu E_1(t), \label{e0}\\[5pt]
		\frac{1}{2}\frac{d}{dt}E_1(t) &= -\nu E_2(t) - \xi(t)E_3(t), \label{e1}\\[5pt]
		\frac{1}{2}\frac{d}{dt}E_3(t) &= -\tfrac{1}{2}\xi(t)k^2E_4(t) \;-\; \nu \Re\langle ik\partial_y v\partial_y \theta, \partial_y^2\theta\rangle \;-\; \tfrac{1}{2}\nu \Re\langle ik\partial_y^2v\theta, \partial_y^2\theta\rangle, \label{e3}\\[5pt]
		\frac{1}{2}\frac{d}{dt}E_4(t) &= -\nu \|\partial_y v\partial_y\theta\|_{L^2_y}^{2}  \;-\; 2\nu \Re\langle \partial_y v\partial_y^2v\theta, \partial_y\theta\rangle. \label{e4}
\end{align} 
We now assume Proposition \ref{P:func_decay} to hold and prove Theorem \ref{T:hyp_weights_main}.
\begin{proof}[Proof of Theorem \ref{T:hyp_weights_main}]
	Following the arguments for the proof in \cite[~Corollary 3.3]{cotizelatigallay}, 
	we introduce the timescale
	\begin{align}
		T_{\nu,k}=\frac{1}{\nu^{s^*} \abs{k}^{1/2}},
	\end{align}
	where $s^*=\min\{\frac{1}{2},s\}$.  The reason for the introduction of this auxiliary $s^*$ is purely technical, as it allows a connection to the autonomous setting in the formal limit $s \rightarrow \infty$. It can be avoided if one restricts to $s\leq \frac{1}{2}$.\\
	Upon integration of the energy estimate \eqref{e0}, the mean-value theorem ensures the existence of a time $t_0 \in (0,  T_{\nu,k})$, such that
	\begin{align}
		2E_1(t_0) = \frac{E_0 (0) - E_0 (T_{\nu,k})}{\nu T_{\nu,k}} \leq \nu^{s^*-1}\abs{k}^{1/2} E_0(0) .
	\end{align}
	This leads to the following bound
	\begin{align}
		\alpha_0 (w_{\nu}(t_0))^{3}E_1(t_0)\leq \frac{\nu^{s^*-1/2}\beta^{1/2}(w_\nu(t_0))^{3}}{2}  E_0(0),
	\end{align}
	where we have used the fact that we defined $\alpha_0=(\beta_0 \nu)^{1/2}$ and $\beta_0=\beta \abs{k}^{-1}$ (see\eqref{beta_def}).  By \eqref{coercivity} we then have the estimate
	\begin{align}
		\Phi(t_0) &\leq \frac{1}{8}[4 E_0(t_0)+5\alpha_0 (w_{\nu}(t_0))^{3}E_1(t_0)+5\gamma_0 k^2 (w_{\nu}(t_0)) E_4(t_0)]\\
		&\leq K_0 (1 + \nu^{s^*-1/2} +\gamma_0 k^2)E_0(0),
	\end{align}
	where we used again \eqref{e0} to bound $E_4(t_0)\leq \|\partial_y v\|_{L^{\infty}}^2 E_0(t_0)< \|\partial_y v\|_{L^{\infty}}^2E_0(0)$ and that the weight function $w_{\nu}$ is bounded from above by a constant. We emphasize that the constant $K_0> 0$ is independent of $\nu$ and $k$, but it does depend on $\beta$ and $\lVert \partial_y v \rVert_{L^\infty}^2$. Then, for times $t\in (t_0,T]$, we apply Proposition \ref{P:func_decay} to find
	\begin{align}
		\frac{1}{2}E_0(t) 
		&\leq \Phi(t) \leq
		\exp\left(-\frac{1}{4}(\beta\nu\abs{k})^{\frac 12} \int_{t_0}^t \xi(\tau)^{\frac 12}w_{\nu}(\tau)\, d\tau\right) \Phi(t_0)
	\end{align}
	and by the above estimate for $\Phi(t_0)$ we then have
	\begin{multline*}
		\frac{1}{2}E_0(t) \leq K_0 (1 + \nu^{s^*-1/2}+ \gamma_0 k^2)\exp\left(\frac{1}{4}(\beta\nu\abs{k})^{\frac 12} \int_0^{t_0} \xi(\tau)^{\frac 12}w_{\nu}(\tau)\, d\tau\right) \\ \qquad \qquad\qquad\times \exp\left(-\frac{1}{4}(\beta\nu\abs{k})^{\frac 12} \int_0^t \xi(\tau)^{\frac 12}w_{\nu}(\tau)\, d\tau\right)  E_0 (0)\,,
	\end{multline*}
	where we wrote $\int_{t_0}^{t}=\int_0^{t}-\int_{0}^{t_0}$.
	
	Now note that for $t\leq t_0$ we have the bound
	
	\begin{align*}
		w_{\nu}(t)\xi(t)^{\frac 12}\leq w_{\nu}(t)^{1-\frac{\ell}{2}}\leq 2^{\frac{1}{2}\ell-1} \quad \mbox{ for } \ell> 2,
	\end{align*}
	where we used that $t_0 \leq T_{\nu, \ell}=\frac{1}{\nu^{s^*}|k|^{\frac 12}}$ and $|k|\geq 1$. 
	By this upper bound on $t_0$ we conclude that
	\begin{multline*}
		\frac{1}{2}E_0(t) \leq K_0 (1 + \nu^{s^*-1/2} + \gamma_0 k^2)  \exp\left(\frac{1}{4} \beta^{\frac 12}\nu^{\frac 12-s^{\ast}}2^{\frac 12 \ell-1} \right)\\
		\times \exp\left(-\frac{1}{4}(\beta\nu\abs{k})^{\frac 12} \int_0^t \xi(\tau)^{\frac 12}w_{\nu}(\tau)\, d\tau\right)  E_0 (0).
	\end{multline*}
	Now recalling our choice of $\gamma_0$ in \eqref{beta_def} we obtain
		\begin{multline*}
			\frac{1}{2}E_0(t) \leq K_0 \qty(1 + \nu^{s^*-1/2} + \qty(\frac{\beta^3\abs{k}}{\nu})^{1/2} )\exp\left(\frac{1}{4} \beta^{\frac 12}\nu^{\frac 12-s^{\ast}}2^{\frac 12 \ell-1} \right)\\
			\times \exp\left(-\frac{1}{4}(\beta\nu\abs{k})^{\frac 12} \int_0^t \xi(\tau)^{\frac 12}w_{\nu}(\tau)\, d\tau\right)  E_0 (0).
		\end{multline*}
		The first exponential factor is simply a constant. Recalling that $\nu \ll 1$, the factor $\nu^{1/2 - s^*}$ in its exponent poses no issue since $s^* \leq \frac{1}{2}$. Similarly, the second prefactor term $\nu^{s^* - 1/2}$ is at most (when $s^* = 0$) of order $\mathcal{O}(\nu^{-1/2})$ and thus can be absorbed into the third prefactor term. Lastly, we recall that $\beta \leq 1$. Up to adjusting the constant $K_0$ we obtain
		\begin{align}
			E_0(t) \leq C_{\rm{ed}} \qty(1 + \qty(\frac{\abs{k}}{\nu})^{1/2})   \exp\left(-\frac{1}{4}(\beta\nu\abs{k})^{\frac 12} \int_0^t \xi(\tau)^{\frac 12}w_{\nu}(\tau)\, d\tau\right)  E_0 (0).
		\end{align}
		The final statement is then achieved by applying the transformation \eqref{transformation}.
	
\end{proof}

\subsection{Proof of Proposition \ref{P:func_decay}}
\label{Proof:general}
Before outlining the strategy used in the proof of Proposition \ref{P:func_decay}, we first recall a spectral-gap type inequality as written in \cite{Bedrossian2017} (see also \cite{cotizelatigallay}).

\begin{lemma}[Proposition 2.7 in \cite{Bedrossian2017}]
	\label{L:spectral_gap}
	Suppose that $v:\mathbb{T}\to\R$ satisfies the assumptions of Theorem \ref{T:hyp_weights_main}. Then there exists a constant $C_{sp} \geq 1$ such that for all $\sigma \in (0,1]$,
	\begin{equation}
		\sigma^{\tfrac{1}{2}} E_0 \lesssim C_{sp} [\sigma E_1 + E_4]\,.
	\end{equation}
\end{lemma}

We will frequently use this inequality. For its proof, we refer the reader to \cite[Proposition 3.7]{cotizelatigallay}.

The proof of Proposition \ref{P:func_decay} relies on the following key lemma:

\begin{lemma}
	\label{L:functionaldecayprop}
	Under the same assumptions as in Proposition \ref{P:func_decay}, the functional $\Phi$ satisfies, for any finite time $T$,
	\begin{equation}\label{functional_final}
		\frac{d}{dt}\Phi + \frac{1}{4}(\beta\,\xi(t)\,\nu |k|)^{\tfrac{1}{2}} w_{\nu}(t)\,\Phi \leq 0, \quad \forall t \leq T.
	\end{equation}
\end{lemma}

Applying Gronwall's inequality to \eqref{functional_final} immediately yields Proposition \ref{P:func_decay}.

\begin{proof}
	We divide the proof in three steps: in the first one we will use energy balances to derive an upper bound on $\Phi$, under specific growth conditions on $\xi(t)$. In the second step, we apply the spectral gap estimate of Lemma \ref{L:spectral_gap} in order to “eliminate” $E_0$ from the upper bound. In the third step, we reconstruct the functional.
	
	\textbf{Step 1: Functional estimate}
	
	We first compute the time derivative of the functional $\Phi$ and find
	\begin{align}
		\label{phi_dv_t}
		\dv{}{t}\Phi= \frac{1}{2}\dv{}{t}E_0 + \frac{1}{2}\alpha_0 \dv{}{t} \qty(w_{\nu}^{3} E_1)+ \beta_0 \dv{}{t}\qty(w_{\nu}^{2}E_3)+\frac{1}{2}\gamma_0 k^2 \dv{}{t}\qty(w_{\nu} E_4).
	\end{align}

	Thanks to the identities \eqref{e0}-\eqref{e4}, \eqref{phi_dv_t} can be rewritten as follows
	\begin{align*}
		&\dv{}{t}\Phi - \frac{1}{2}\alpha_0  E_1 \dv{w_{\nu}^{3}}{t} -\beta_0 E_3 \dv{w_{\nu}^{2}}{t} -\frac{1}{2}\gamma_0 k^2 E_4 \dv{w_{\nu}}{t} \\ &+\nu E_1 + \nu \alpha_0 w_{\nu}^{3} E_2+\beta_0 \xi(t)k^2w_{\nu}^{2}E_4 + \nu\gamma_0 k^2 w_{\nu}\norm{\partial_y v\partial_y\theta}_{L^2_y}^{2} \\
		=&-\alpha_0 w_{\nu}^{3} \xi(t)E_3-2\nu\beta_0w_{\nu}^{2}\Re \inner{ik\partial_y v\partial_y\theta}{\partial_y^2 \theta}-\nu\beta_0w_{\nu}^{2}\Re\inner{ik\partial_y^2v \theta}{\partial_y^2 \theta}\\
		& -2\nu\gamma_0 k^2 w_{\nu}\Re\inner{\partial_y v\partial_y^2 v \theta}{\partial_y\theta}.
	\end{align*}
	We first estimate the inner products on the right-hand side by Young's inequality and get
	\begin{align*}
		\lvert \alpha_0 \xi(t) w_{\nu}^{3}E_3 \rvert
		&\leq \frac{1}{2}\frac{\alpha_0^2}{\beta_0}\xi(t)w_{\nu}^{4}E_1+\frac{k^2}{2}\beta_0\xi(t) w_{\nu}^{2}E_4, \\
		\lvert 2\nu \beta_0 w_{\nu}^{2}\Re\inner{ik \partial_y v\partial_y \theta}{\partial_y^2\theta} \rvert
		&\leq \frac{\nu}{2}\alpha_0 w_{\nu}^{3}E_2+2\nu k^2 \frac{\beta_0^2}{\alpha_0} w_{\nu} \norm{\partial_y v \partial_y \theta}_{L^2_y}^{2} , \\
		\lvert \nu \beta_0 w_{\nu}^{2}\Re \inner{ik \partial_y^2v\theta}{\partial_y^2\theta} \rvert
		&\leq \frac{\nu}{2}\alpha_0 w_{\nu}^{3}E_2+\frac{\nu k^2}{2}\frac{\beta_0^2}{\alpha_0}w_{\nu}\norm{\partial_y^2v\theta}_{L^2_y}^{2} , \\
		\lvert 2\nu\gamma_0 k^2 w_{\nu}\Re\inner{\partial_y v\partial_y^2 v \theta}{\partial_y\theta} \rvert &\leq \frac{\nu k^2}{2}\gamma_0 w_{\nu} \lVert \partial_y v \partial_y \theta \rVert_{L^2_y}^2 +2 \gamma_0 \nu k^2w_{\nu}\norm{\partial_y^2v\theta}_{L^2_y}^{2} .
	\end{align*}
	Using these inequalities we can obtain the following estimate for the time derivative of the functional \eqref{phi_dv_t}
	\begin{align*}
		&\dv{}{t}\Phi  - \frac{1}{2}\alpha_0  E_1 \dv{w_{\nu}^{3}}{t} -\beta_0 E_3 \dv{w_{\nu}^{2}}{t} -\frac{1}{2}\gamma_0 k^2 E_4 \dv{w_{\nu}}{t} \\
		+&\qty(\nu-\frac{1}{2}\frac{\alpha_0^2}{\beta_0}\xi(t)w_{\nu}^{4})E_1+\frac{k^2}{2}\beta_0\xi(t) w_{\nu}^{2}E_4\\
		\leq&\nu k^2 \qty(2\frac{\beta_0^2}{\alpha_0}-\frac{\gamma_0}{2})w_{\nu}\norm{\partial_y v\partial_y\theta}_{L^2_y}^{2} +\nu k^2 w_{\nu}\qty(\frac{1}{2}\frac{\beta_0^2}{\alpha_0}+2\gamma_0)\norm{\partial_y^2v\theta}_{L^2_y}^{2} .
	\end{align*}
	Now by using condition \eqref{condition_param_2} this estimate can be simplified to 
	\begin{multline*}
		\dv{}{t}\Phi - \frac{1}{2}\alpha_0  E_1 \dv{w_{\nu}^{3}}{t} -\beta_0 E_3 \dv{w_{\nu}^{2}}{t} -\frac{1}{2}\gamma_0 k^2 E_4 \dv{w_{\nu}}{t} 
		+\qty(\nu-\frac{1}{2}\frac{\alpha_0^2}{\beta_0}\xi(t)w_{\nu}^{4})E_1\\+\frac{k^2}{2}\beta_0\xi(t) w_{\nu}^{2}E_4
		\leq\nu k^2c_\infty \gamma_0w_{\nu} E_0,
	\end{multline*}
	where we have used the constraint $ 3 \norm{\partial_y^2v}^2_{L^\infty}\leq c_\infty$.
	Next we wish to estimate the terms with time derivatives of the weights and therefore explicitly compute them by using relation \eqref{weight_derivative}
	\begin{multline*}
		\dv{}{t}\Phi + \alpha_0 \frac{3 \nu^s}{2}w_{\nu}^{4} E_1+\frac{\nu^s}{2}\gamma_0 k^2w_{\nu}^{2}E_4+
		\qty(\nu-\frac{1}{2}\frac{\alpha_0^2}{\beta_0}\xi(t)w_{\nu}^{4})E_1+\frac{k^2}{2}\beta_0\xi(t) w_{\nu}^{2}E_4\\
		\leq- 2\beta_0\nu^sw_{\nu}^{3} E_3+\nu k^2c_\infty \gamma_0w_{\nu} E_0.
	\end{multline*}
	Similarly to the estimates above, by Young's inequality we find 
	\begin{align*}
		\lvert 2\beta_0 \nu^sw_{\nu}^{3} E_3 \rvert &\leq \frac{\alpha_0 }{2}\nu^sw_{\nu}^{4}E_1+\frac{2\beta_0^2}{ \alpha_0}\nu^s k^2 w_{\nu}^{2}E_4\\
		&\leq\frac{\alpha_0 }{2}\nu^sw_{\nu}^{4}E_1+\frac{k^2}{8}  \gamma_0 \nu^s k^2w_{\nu}^{2}E_4,
	\end{align*}
	where we again used condition \eqref{condition_param_2} in the second line.
	
	Now we can absorb these terms into their respective counterparts on the left and subsequentially drop the residuals of these terms. This simplifies the estimate for $\dv{}{t}\Phi$ further to
	\begin{equation*}
		\dv{}{t}\Phi + \qty(\nu-\frac{1}{2}\frac{\alpha_0^2}{\beta_0}\xi(t)w_{\nu}^{4})E_1+\frac{k^2}{2}\beta_0\xi(t) w_{\nu}^{2}E_4
		\leq\nu k^2c_\infty \gamma_0w_{\nu} E_0.
	\end{equation*}
	\\ Then, the upper bound in \eqref{gamma_bounds} and the choice 
	\begin{equation}\label{choice-alpha0}
		\alpha_0=\nu^{\frac 12}\beta_0^{\frac 12}\,,
	\end{equation}
	imply
	\begin{equation*}
		\nu-\frac{1}{2}\frac{\alpha_0^2}{\beta_0}\xi(t)w_{\nu}^{4}\geq \frac{\nu}{4}\xi(t)w_{\nu}^\ell.
	\end{equation*}
	
	Thus, we obtain
	\begin{equation}\label{temp1}    
		\dv{}{t}\Phi+\frac{\nu}{4}\left(1+\xi(t)w_{\nu}^{\ell}\right)E_1+\frac{k^2}{2}\beta_0 \xi(t) w_{\nu}^{2}E_4\leq \nu k^2 \gamma_0 c_\infty w_{\nu} E_0.
	\end{equation}

	\textbf{Step 2: Spectral gap estimate}\\
	
	We start from \eqref{temp1} and rearrange it as follows:
	\begin{multline*}   
		\dv{}{t}\Phi+\frac{\nu}{4}\xi(t)w_{\nu}^{\ell}E_1+\frac{k^2}{4}\beta_0 \xi(t) w_{\nu}^{2}E_4+A\nu k^2 \gamma_0 c_\infty w_{\nu} E_0\\
		+\frac{\nu}{4}E_1+\frac{k^2}{4}\beta_0 \xi(t) w_{\nu}^{2}E_4
		\leq (A+1)\nu k^2 \gamma_0 c_\infty w_{\nu} E_0\,,
	\end{multline*}
	where $A>0$ is a balancing parameter, to be determined. Our goal is to control the bad term caused by diffusion on the right-hand side of the inequality with $E_1$ and $E_4$. Now we recall the spectral gap estimate from Lemma \ref{L:spectral_gap}, which we rewrite in the form
	\begin{equation}
		E_0\leq C_{sp}[\sigma^{\frac 12} E_1+\sigma^{-\frac 12}E_4],\quad \forall \sigma\in (0,1].
	\end{equation}
	If we choose 
	$$\sigma= \frac{\nu}{k^2\beta_0 \xi(t)w_{\nu}(t)^2}\leq 1 \qquad \mbox{ and } A=\frac{1}{4}\frac{\beta_0^{\frac 12}\xi(t)^{\frac 12}\nu^{-\frac 12}}{\abs{k}\gamma_0 c_{\infty}C_{\rm{sp}}}-1\,.$$
	then we absorb the right-hand side within parts of the left-hand side of the inequality involving $\Phi$:
	\begin{equation*}   
		\dv{}{t}\Phi+\frac{\nu}{4}\xi(t)w_{\nu}^{\ell}E_1+\frac{k^2}{4}\beta_0 \xi(t) w_{\nu}^{2}E_4+A\nu k^2 \gamma_0 c_\infty w_{\nu} E_0\leq 0.
	\end{equation*}
	On the one hand, we notice that the spectral-gap constraint $\sigma \leq 1$ translates in the time-dependent lower bound for $\xi(t)$
	\begin{equation}
		\label{gamma_bound-dynamic}
		\xi(t)\geq \frac{\nu}{k^2 \beta_0 w_{\nu}(t)^2}=\frac{\nu}{k^2 \beta_0 }(1+\nu^s t)^2
	\end{equation}
	
	On the other hand, we observe that if \begin{align}
		\label{gamma_bound-static}
		\xi(t)\geq 25\frac{k^2\gamma_0^2\nu c_{\infty}^2C_{\rm{sp}}^2}{\beta_0}
	\end{align}
	then, without affecting scaling behaviours, we can estimate the above inequality and arrive at
	\begin{equation*}  
		\dv{}{t}\Phi+\frac{\nu}{4}\xi(t)w_{\nu}^{\ell}E_1+\frac{k^2}{4}\beta_0 \xi(t) w_{\nu}^{2}E_4+ \frac{1}{4}(\beta_0\xi(t) \nu)^{\frac 12}\abs{k} w_{\nu}\ E_0\leq 0\,.
	\end{equation*}
	Towards the reconstruction of the functional, we rewrite this as
	\begin{multline*}   
		\dv{}{t}\Phi+ \frac{1}{4}(\beta_0\xi(t) \nu)^{\frac 12}\abs{k} w_{\nu}\qty[E_0+\frac{\nu^{1/2}}{\beta_0^{1/2}\abs{k}}\xi(t)^{1/2}w^{\ell-1}E_1+\frac{\abs{k}}{\nu^{1/2}}\beta_0^{1/2} \xi(t)^{1/2} w_{\nu} E_4] \\ \leq 0.
	\end{multline*}
	Now since $w_{\nu}^{\ell-1}\geq w_{\nu}^3$ for $\ \ell\leq 4$ we can already restore the proper weights of the functional:
	\begin{equation}  \label{functional_post-spectral}  
		\dv{}{t}\Phi+ \frac{1}{4}(\beta_0\xi(t)\nu)^{\frac 12} \abs{k}w_{\nu}\qty[E_0+\frac{\nu^{1/2}}{\beta_0^{1/2}\abs{k}}\xi(t)^{1/2}w_{\nu}^3 E_1+\frac{\abs{k}}{\nu^{1/2}}\beta_0^{1/2} \xi(t)^{1/2} w_{\nu} E_4] \leq 0.
	\end{equation}
	\textbf{Step 3: Rebuilding the functional}\\
	We are left with restoring the correct parameters of the functional and dealing with the residual powers of $\xi(t)$ in front of $E_1$ and $E_4$.
	
	Recalling the choice of $\gamma_0$ in \eqref{beta_def}, we must have
	\begin{equation}
		\gamma_0=16\frac{\beta_0^{3/2}}{\nu^{1/2}}
	\end{equation}
	and inserting this choice into \eqref{gamma_bound-static} we get 
	\begin{align*}
		\xi(t) \geq 6400 k^2\beta_0^2 c_\infty^2 C_{sp}^2.
	\end{align*}
	Further, the lower bound is uniform in $k$, as we recall from \eqref{beta_def} that 
	\begin{equation}\label{choice-beta0}
		\beta_0=\beta \abs{k}^{-1}, \quad \mbox{for some } \beta\in \qty[\frac{\nu}{\abs{k}},1]
	\end{equation}
	and hence set $80 c_\infty C_{sp}=: C $.
	Now we can rewrite \eqref{gamma_bound-dynamic} as \eqref{gamma_bound-static} as
	\begin{equation}
		\label{gamma_bounds-simple-d}
		\xi(t) \geq \frac{\nu}{\abs{k}\beta}(1+\nu^st)^{2}, 
	\end{equation}
	and 
	\eqref{gamma_bound-static} as
	\begin{equation}
		\label{gamma_bounds-simple-s}
		\xi(t) \geq C^2\beta^2.
	\end{equation}
	
	Depending on the dominating lower bounds, we identify two regimes: 
	$$ t\leq t^{\ast}_{\nu,k}\quad \mbox{ and }\quad  t\geq t^{\ast}_{\nu,k} \qquad \mbox{ where } \quad t^{\ast}_{\nu,k}=\frac{1}{\nu^s}\left(\frac{|k|^{\frac 12}\nu^{\frac 12}C\beta^{\frac 32}}{\nu^{\frac 12}}-1\right)$$
	
	\textbf{Short-time estimate}\\
	Assume $t < t_{\nu,k}^*$. Since this is the time interval where  the lower bound in \eqref{gamma_bounds-simple-s} dominates over the lower bound in \eqref{gamma_bounds-simple-d}, we find
	\begin{equation*}   
		\dv{}{t}\Phi+ \frac{1}{4}(\beta\xi(t)\nu\abs{k})^{\frac 12} w_{\nu}\qty[E_0+\frac{\nu^{1/2} \beta^{1/2}}{\abs{k}^{1/2}}Cw_{\nu}^3 E_1+\frac{\abs{k}^{1/2}}{\nu^{1/2}}\beta^{3/2}C w_{\nu} E_4] \leq 0.
	\end{equation*}
	Recalling our parameter choices \eqref{beta_def} this leads to
	\begin{equation*}   
		\dv{}{t}\Phi+ \frac{1}{4}(\beta\xi(t)\nu\abs{k})^{\frac 12} w_{\nu}\qty[E_0+C\alpha_0 w_{\nu}^3 E_1+\frac{1}{16}k^{2}C\gamma_0 w_{\nu} E_4] \leq 0.
	\end{equation*}
	Notice that, by a suitable parameter choice, we can always make sure that
	$\frac{1}{16}C=5c_\infty C_{sp}\geq 5$ holds. 
	Hence, we get
	\begin{equation} 
		\label{short_time-final_1}
		\dv{}{t}\Phi+ \frac{1}{4}(\beta\xi(t)\nu\abs{k})^{\frac 12} w_{\nu}\qty[E_0+5\alpha_0 w_{\nu}^3 E_1+5 k^2 \gamma_0 w_{\nu} E_4] \leq 0.
	\end{equation}
	\\
	\textbf{Large-time estimate}\\
	Assume $t \geq t^*$. In this case,
	$\xi(t)\geq \frac{\nu} {|k|\beta}w_{\nu}(t)^{-2}$ and we can insert this lower bound in \eqref{functional_post-spectral} to obtain
	\begin{align*}   
		\dv{}{t}\Phi+ \frac{1}{4}(\beta\xi(t)\nu\abs{k})^{\frac 12} w_{\nu}\bigg[E_0&+\frac{\nu^{1/2}}{\beta^{1/2}\abs{k}^{1/2}}\qty(\frac{\nu}{\abs{k}\beta})^{1/2}w_{\nu}(t)^{-1}w_{\nu}^3 E_1\\
		&+\frac{\abs{k}^{1/2}}{\nu^{1/2}}\beta^{1/2} \qty(\frac{\nu}{\abs{k}\beta})^{1/2}w_{\nu}(t)^{-1} w_{\nu} E_4\bigg] \leq 0.
	\end{align*}
	We now recall that, in this regime, $\frac{\nu} {|k|\beta}w_{\nu}(t)^{-2}\geq C^2\beta^2$, implying
	\begin{align*}
		w_{\nu}^{-1}\geq \qty(\frac{\abs{k}\beta}{\nu})^{1/2}C
	\end{align*}
	such that 
	\begin{equation*}   
		\dv{}{t}\Phi+ \frac{1}{4}(\beta\xi(t)\nu\abs{k})^{\frac 12} w_{\nu}\qty[E_0+\frac{\nu^{1/2}\beta^{1/2}}{\abs{k}^{1/2}}Cw_{\nu}^3 E_1+\frac{\abs{k}^{1/2}}{\nu^{1/2}}\beta^{3/2} Cw_{\nu} E_4] \leq 0.
	\end{equation*}
	Inserting the parameter choices \eqref{beta_def} we obtain
	\begin{equation*}   
		\dv{}{t}\Phi+ \frac{1}{4}(\beta\xi(t)\nu\abs{k})^{\frac 12} w_{\nu}\qty[E_0+C\alpha_0 w_{\nu}^3 E_1+k^2 C\gamma_0 w_{\nu} E_4] \leq 0.
	\end{equation*}
	Therefore, arguing in the same manner as above,
	\begin{equation} 
		\label{short_time-final}
		\dv{}{t}\Phi+ \frac{1}{4}(\beta\xi(t)\nu\abs{k})^{\frac 12} w_{\nu}\qty[E_0+5\alpha_0 w_{\nu}^3 E_1+5 k^2 \gamma_0 w_{\nu} E_4] \leq 0.
	\end{equation}
	Lastly, by applying the coercivity of the energy functional given in \eqref{coercivity} to both estimates we arrive at the statement of the Lemma. 
\end{proof}
Proposition \ref{P:func_decay} is then directly implied by Lemma \ref{L:functionaldecayprop}.

\section{Enhanced dissipation for On/Off-velocity fields}\label{S:inter}

In this section, we treat velocity fields that are active at certain times and might be turned off at others. As remarked in the introduction, this class of stirring fields is especially important in applications. \\
We proceed under the same assumptions outlined in Section \ref{Section-product}, specified in \eqref{setup_1}-\eqref{setup2}, and continue our analysis of the advection-diffusion equation presented in \eqref{Heu:adv-eq}.

The starting point is once again an augmented energy functional, albeit with a different modification. It takes the form of
\begin{equation}
	\Psi(t) \coloneqq \frac 12\left[E_0(t)+\alpha_0(t) E_1(t)+2\beta_0(t) E_3(t)+\gamma_0(t)  k^2E_4(t)\right],
\end{equation}
where $E_0,E_1, E_3, E_4$ are defined in \eqref{Es}.
The choice of this specific form of the functional $\Psi$ will become clear in the arguments that follow.
For the moment, we only notice two differences
with respect to the energy functional $\Phi$ defined in \eqref{aug-functional}: the absence of time-dependent weights and the time-dependency of the balancing parameters $\alpha_0,\beta_0$ and $\gamma_0$.
The second will be necessary in order to balance the scaling of the various error terms arising in those cases when $\xi(t)$ becomes very small, i.e. scaling with powers of $\nu$. This is crucial as it allows us to avoid lower bounds on $\xi(t)$ and therefore let it go to zero.

We assume that for all times $t\geq 0$ the non-negative parameters $\alpha_0,\beta_0, \gamma_0$ are bounded, and we impose the equivalent of condition \eqref{condition_param_2} for time-dependent parameters:
\begin{align}
	\label{O:condition_param_2}
	\beta_0(t)^2\leq \frac{1}{16}\gamma_0(t)\alpha_0(t),
\end{align} 
Then in analogy with the analysis in Section \ref{Section-product} we define 
	\begin{align}
		\label{O:beta_def}
		\alpha_0(t)=\nu^{1/2} \beta_0(t)^{1/2}, \quad \beta_0(t)=\frac{\beta(t)}{\abs{k}}, \quad\gamma_0(t)=16\frac{\beta_0(t)^{3/2}}{\nu^{1/2}},
	\end{align}
	where $\beta(t) \leq 1$ is now a function of time and again to be fixed later.\\
\\
Note that the coercivity of the augmented energy functional $\Psi$  follows from the point-wise in time bound \eqref{O:condition_param_2}.  By the same strategy employed to derive \eqref{coercivity}, we find
\begin{multline}
	\label{O:coercivity}
	\frac{1}{8}\qty(4E_0+3\alpha_0(t) E_1(t)+3\gamma_0(t) k^2 E_4(t))\leq \Psi(t)\\
	 \leq\frac{1}{8}\qty(4E_0(t)+5\alpha_0(t) E_1(t)+5\gamma_0(t) k^2 E_4(t)).
\end{multline}\\
In order to prove Theorem \ref{T:on_off} we first need to derive a quantitative bound for the functional $\Psi(t)$. This is contained in the next proposition.

\begin{proposition}\label{O:P:func_decay}
	Assume that $\xi(t)$ satisfies the conditions
	\begin{align}
		\xi(t) \leq 1, \quad \xi'(t) \leq{C_{\xi}}(\nu\abs{k})^{1/2},
	\end{align}
	for some small constant $C_{\xi}$.
	Then there exists a constant $C_\xi'$ that may only depend on $\norm{\partial_y^2v(y)}_{L^\infty}$ such that the energy functional $\Psi$ satisfies the following estimate
	\begin{align}
		\label{O:prop_1_statement}
		\Psi(t)\leq \exp \qty(-C_\xi'(\nu \abs{k})^{1/2}\int_0^t \xi(\tau)^3 d\tau) \Psi(0)
	\end{align}
	for any $t\leq T$ where the time $T$ is arbitrary but finite. 
\end{proposition}
We now assume Proposition \ref{O:P:func_decay} and prove Theorem \ref{T:on_off}.
\begin{proof}[Proof of Theorem \ref{T:on_off}]
	We introduce the timescale
	\begin{align}
		T_{\nu,k}=\frac{1}{\nu^{1/2} \abs{k}^{1/2}},
	\end{align}
	Upon integrating the energy estimate \eqref{e0}, the mean-value theorem ensures the existence of a time $t_0 \in (0,  T_{\nu,k})$, such that
	\begin{align}
		2E_1(t_0) = \frac{E_0 (0) - E_0 (T_{\nu,k})}{\nu T_{\nu,k}} \leq \nu^{-1/2}\abs{k}^{1/2} E_0(0) .
	\end{align}
	This leads to the following bound
	\begin{align}
		\alpha_0(t_0) E_1(t_0)\leq \frac{\beta(t_0)^{1/2}}{2}  E_0(0),
	\end{align}
	where we have used the definitions \eqref{O:beta_def}.\\
	By \eqref{O:coercivity} we then have the estimate
	\begin{align*}
		\Psi(t_0) &\leq \frac{1}{8}[4 E_0(t_0)+5\alpha_0(t_0) E_1(t_0)+5\gamma_0(t_0) k^2 E_4(t_0)]\\
		&\leq K_0 (1 +\beta(t_0)^{1/2}+\gamma_0(t_0) k^2)E_0(0),
	\end{align*}
	where we used again \eqref{e0} to bound $E_4(t_0)\leq \|\partial_y v\|_{L^{\infty}}^2 E_0(t_0)< \|\partial_y v\|_{L^{\infty}}^2E_0(0)$. We emphasize that the constant $K_0> 0$ is independent of $\nu$ and $k$, but it does depend on $\lVert \partial_y v \rVert_{L^\infty}^2$. Then, for times $t\in (t_0,T]$, we apply Proposition \ref{P:func_decay} to find
	\begin{align}
		\frac{1}{2}E_0(t) 
		&\leq \Psi(t) \leq
		\exp\left(-C_\xi'(\nu\abs{k})^{\frac 12} \int_{t_0}^t \xi(\tau)^3\, d\tau\right) \Psi(t_0)
	\end{align}
	and by the above estimate for $\Psi(t_0)$ we then have
	\begin{multline*}
		\frac{1}{2}E_0(t) \leq K_0 (1 + \beta(t_0)^{1/2}+ \gamma_0(t_0) k^2)\exp\left(C_\xi'(\nu\abs{k})^{\frac 12} \int_0^{t_0} \xi(\tau)^3\, d\tau\right) \\ \qquad \qquad\qquad\times \exp\left(-C_\xi'(\nu\abs{k})^{\frac 12} \int_0^t \xi(\tau)^3\, d\tau\right)  E_0 (0)\,,
	\end{multline*}
	where we wrote $\int_{t_0}^{t}=\int_0^{t}-\int_{0}^{t_0}$. \\
	Using that $t_0\leq T_{\nu, k}=(\nu|k|)^{-\frac 12}$ we get
	\begin{multline}
		\label{O:L2_decay}
		\frac{1}{2}E_0(t) \leq K_0 (1 + \beta(t_0)^{1/2}+ \gamma_0(t_0) k^2)  \exp\left(C_\xi'\right)\\\times \exp\left(-C_\xi'(\nu\abs{k})^{\frac 12} \int_0^t \xi(\tau)^3\, d\tau\right)  E_0 (0).
	\end{multline}
	Now recalling our choice of $\gamma_0$ in \eqref{O:beta_def} we obtain
		\begin{multline*}
			\frac{1}{2}E_0(t) \leq K_0 \qty(1 + \beta(t_0)^{1/2} + \qty(\frac{\beta(t_0)^3\abs{k}}{\nu})^{1/2} )\exp\left(C_\xi'\right)\\
			\times \exp\left(-C_\xi'(\nu\abs{k})^{\frac 12} \int_0^t \xi(\tau)^3\, d\tau\right)  E_0 (0).
		\end{multline*}
		The first exponential factor is simply a constant and we recall that $\beta(t) \leq 1$. Up to adjusting the constant $K_0$ we obtain
		\begin{align}
			E_0(t) \leq C_{\rm{ed}} \qty(1 + \qty(\frac{\abs{k}}{\nu})^{1/2})   \exp\left(-C_\xi'(\nu\abs{k})^{\frac 12} \int_0^t \xi(\tau)^3\, d\tau\right)  E_0 (0).
		\end{align}
		The final statement is then achieved by applying the transformation \eqref{transformation}.
	
\end{proof}
\subsection{Proof of Proposition \ref{O:P:func_decay}}

The proof of Proposition \ref{O:P:func_decay} is a direct consequence of the following 
\begin{lemma} \label{O:functionaldecayprop}
	Under the same assumptions of Proposition \eqref{O:P:func_decay}, the functional $\Psi$ satisfies the following estimate
	\begin{align}
		\label{O:functional_final}
		\dv{}{t}\Psi+C_\xi'(\nu \abs{k})^{1/2}\xi(t)^3 
		\Psi &\leq 0,\quad \forall t \leq T,
	\end{align} 
	where time $T$ is arbitrary but finite.
\end{lemma}
This Lemma, when combined with an application of Grönwall's inequality, implies Proposition \ref{O:P:func_decay}. 
\begin{proof}
	In structural analogy to the proof of Lemma \ref{functional_final}, we divide the proof in three steps: in the first one we will use energy
	balances to derive an upper bound on $\Psi$, where we now need to control additional error terms arising from the time-dependency of $\beta$. In the second step, we apply the spectral gap estimate of Lemma \ref{L:spectral_gap} in
	order to ”eliminate” $E_0$ from the upper bound. In the third step, we reconstruct the
	functional.\\
	\textbf{Step 1: Functional estimate}
	
	We first compute the time derivative of the functional $\Psi$ and find
	\begin{multline}
		\label{O:phi_dv_t}
		\dv{}{t}\Psi(t)= \frac{1}{2}\dv{}{t}E_0(t) + \frac{1}{2}\dv{}{t} \qty(\alpha_0(t)  E_1(t))\\
		+ \dv{}{t}\qty( \beta_0(t)E_3(t))+\frac{1}{2} k^2 \dv{}{t}\qty(\gamma_0(t)E_4(t)).
	\end{multline}
	Using the above energy identities (see Lemma \ref{L:energy}) this can be rewritten as follows
	\begin{align*}
		&\dv{}{t}\Psi 
		- \frac{1}{2} \dv{\alpha_0 }{t} E_1- \dv{\beta_0}{t}E_3-\frac{1}{2} k^2 \dv{\gamma_0}{t}E_4\\
		&+\nu E_1 + \nu \alpha_0  E_2+\beta_0 \xi(t)k^2E_4 + \nu\gamma_0 k^2 \norm{\partial_y v\partial_y\theta}_{L^2_y}^{2} \\
		=&-\alpha_0  \xi(t)E_3-2\nu\beta_0\Re \inner{ik\partial_y v\partial_y\theta}{\partial_y^2 \theta}-\nu\beta_0\Re\inner{ik\partial_y^2v \theta}{\partial_y^2 \theta}\\
		& -2\nu\gamma_0 k^2 w_\nu \Re\inner{\partial_y v\partial_y^2 v \theta}{\partial_y\theta}.
	\end{align*}
	We estimate the inner products on the right-hand side by means of Young's inequality and get
	\begin{align*}
		\lvert \alpha_0 \xi(t) E_3 \rvert
		&\leq \frac{1}{2}\frac{\alpha_0^2}{\beta_0}\xi(t)E_1+\frac{k^2}{2}\beta_0\xi(t) E_4, \\
		\lvert 2\nu \beta_0 \Re\inner{ik \partial_y v\partial_y \theta}{\partial_y^2\theta} \rvert
		&\leq \frac{\nu}{2}\alpha_0 E_2+2\nu k^2 \frac{\beta_0^2}{\alpha_0}  \norm{\partial_y v \partial_y \theta}_{L^2_y}^{2} , \\
		\lvert \nu \beta_0 \Re \inner{ik \partial_y^2v}{\partial_y^2\theta} \rvert
		&\leq \frac{\nu}{2}\alpha_0 E_2+\frac{\nu k^2}{2}\frac{\beta_0^2}{\alpha_0}\norm{\partial_y^2v\theta}_{L^2_y}^{2} , \\
		\lvert 2\nu\gamma_0 k^2 \Re\inner{\partial_y v\partial_y^2 v \theta}{\partial_y\theta} \rvert &\leq \frac{\nu k^2}{2}\gamma_0  \lVert \partial_y v \partial_y \theta \rVert_{L^2_y}^2 +2 \gamma_0 \nu k^2\norm{\partial_y^2v\theta}_{L^2_y}^{2} ,
	\end{align*}
	from which one obtains the following estimate for the time derivative of the functional \eqref{O:phi_dv_t}
	\begin{align*}
		&\dv{}{t}\Psi 
		- \frac{1}{2} \dv{\alpha_0 }{t} E_1- \dv{\beta_0}{t}E_3-\frac{1}{2} k^2 \dv{\gamma_0}{t}E_4\\
		+&\qty(\nu-\frac{1}{2}\frac{\alpha_0^2}{\beta_0}\xi(t))E_1+\frac{k^2}{2}\beta_0\xi(t) E_4\\
		\leq&\nu k^2 \qty(2\frac{\beta_0^2}{\alpha_0}-\frac{\gamma_0}{2})\norm{\partial_y v\partial_y\theta}_{L^2_y}^{2} +\nu k^2 \qty(\frac{1}{2}\frac{\beta_0^2}{\alpha_0}+2\gamma_0)\norm{\partial_y^2v\theta}_{L^2_y}^{2} .
	\end{align*}
	Now by using conditions \eqref{O:beta_def} and the fact that $\xi(t)\leq 1$, this estimate can be simplified to 
	\begin{align*}
		&\dv{}{t}\Psi 
		- \frac{1}{2} \dv{\alpha_0 }{t} E_1- \dv{\beta_0}{t}E_3-\frac{1}{2} k^2 \dv{\gamma_0}{t}E_4\\
		+&\frac{\nu}{2}E_1+\frac{k^2}{2}\beta_0\xi(t) E_4\\
		\leq& 3\nu k^2 \gamma_0 \norm{\partial_y^2v\theta}_{L^2_y}^{2} \leq\nu k^2c_\infty \gamma_0 E_0,
	\end{align*}
	where we have used the constraint $c_\infty \geq 3 \norm{\partial_y^2v}^2_{L^\infty}$.
	Now we face the issue that the sign of the derivatives of the parameters in the first line is not determined. Hence, we need to be able to control all additional terms.
	In order to do this, we now need to assume a specific form on $\beta$. We choose
	\begin{align}
		\label{O:fix_beta_t}
		\beta(t)= C_\xi\xi(t)^2.
	\end{align}
	The exact form of the constant $C_\xi \leq 1$ will be determined later. 
	
	This choice together with the formulas for the parameters \eqref{O:beta_def} yields
	\begin{align*}
		\dv{}{t}\Psi
		&- \frac{\nu^{1/2} }{2\abs{k}^{1/2}} C_\xi^{1/2} \xi'(t) E_1- 2\abs{k}^{-1}C_\xi\xi'(t) \xi(t)E_3
		-\frac{48}{2\nu^{1/2}}\abs{k}^{1/2}C_\xi^{3/2} \xi'(t)\xi(t)^{2} E_4\\
		&+\frac{\nu}{2}E_1+\frac{k^2}{2}\beta_0\xi(t) E_4
		\leq\nu k^{2}c_\infty \gamma_0 E_0.
	\end{align*}
	Now we consider the error terms of the first line separately to find by Young's inequality that 
	\begin{align*}
		&- \bigg[\frac{\nu^{1/2} }{2} C_\xi^{1/2}\abs{k}^{-1/2} \xi'(t) E_1+ 2C_\xi \abs{k}^{-1}\xi'(t) \xi(t)E_3+\frac{48\abs{k}^{1/2}}{2\nu^{1/2}}  C_\xi^{3/2} \xi'(t)\xi(t)^{2}E_4\bigg]\\
		=&- \bigg[\frac{\nu^{1/2} }{2} C_\xi^{1/2}\abs{k}^{-1/2} \xi'(t) E_1+ 2C_\xi \xi'(t) \xi(t)\Re\langle i\partial_y v\theta, \partial_y \theta\rangle \\
		&\qquad +\frac{48\abs{k}^{1/2}}{2\nu^{1/2}}  C_\xi^{3/2} \xi'(t)\xi(t)^{2}E_4\bigg]\\
		\geq&-\bigg[\frac{5}{2}\nu^{1/2}C_\xi^{1/2}\abs{k}^{-1/2} \xi'(t) E_1+25 \frac{\abs{k}^{1/2}}{\nu^{1/2}} C_\xi^{3/2} \xi'(t)\xi(t)^{2}E_4\bigg].
	\end{align*}
	Then 
	\begin{align*}
		&\dv{}{t}\Psi+
		\left[\frac{\nu}{2}-\frac{5}{2}\nu^{1/2}C_\xi^{1/2}\abs{k}^{-1/2} \xi'(t)\right] E_1+
		\left[\frac{k^2}{2}\beta_0\xi(t)-25 \frac{\abs{k}^{1/2}}{\nu^{1/2}} C_\xi^{3/2} \xi'(t)\xi(t)^{2} \right]E_4\\
		&\leq\nu k^{2}c_\infty \gamma_0 E_0.
	\end{align*}

	implies
	\begin{align}
		\label{O:temp1}
		\dv{}{t}\Psi
		+\frac{\nu}{4}E_1+\frac{k^2}{4}\beta_0\xi(t) E_4
		\leq\nu k^2c_\infty \gamma_0 E_0
	\end{align}
	if the two conditions 
	\begin{eqnarray*}
		\frac{5}{2}\nu^{1/2}C_\xi^{1/2}\abs{k}^{-1/2} \xi'(t)&\leq& \frac{\nu}{4}\\
		25 \frac{\abs{k}^{1/2}}{\nu^{1/2}} C_\xi^{3/2} \xi'(t)\xi(t)^{2}&\leq& \frac{k^2}{4}\beta_0\xi(t)=\frac{k^2}{4}|k|^{-1}C_{\xi}\xi(t)^2\xi(t)
	\end{eqnarray*}
	are simultaneously satisfied. Since $\xi\leq 1$, this is the case when 
	\begin{equation}
		\xi'(t)\leq \frac{1}{100}C_{\xi}^{-\frac 12}|k|^\frac 12\nu^\frac 12\,.
	\end{equation}
	Notice that, for the sake of clarity of the argument, we did not try to optimize on numerical constants at any point.

	\textbf{Step 2: Spectral gap estimate}\\
	We start from \eqref{O:temp1} and rearrange it as follows:
	\begin{multline*}
		\dv{}{t}\Psi
		+\frac{\nu}{8}E_1+\frac{k^2}{8}\beta_0\xi(t) E_4
		+\frac{\nu}{8}\xi(t)^{3} E_1+\frac{k^2}{8}\beta_0\xi(t) E_4+A\nu k^2 c_\infty\gamma_0 E_0  \\ 
		\leq(A+1)\nu k^2c_\infty \gamma_0 E_0
	\end{multline*}
	where $A>0$ is a balancing parameter to be determined, and we used again that $\xi(t) \leq 1$. Now we recall the spectral gap estimate from Lemma \ref{L:spectral_gap}, which we rewrite in the form
	\begin{equation}
		E_0\leq C_{sp}[\sigma^{\frac 12} E_1+\sigma^{-\frac 12}E_4]\quad \forall \sigma\in (0,1].
	\end{equation} 
	If we choose 
	$$\sigma= \frac{\nu}{\abs{k} C_\xi }\leq 1 \qquad \mbox{ and } A=\frac{1}{128 C_\xi c_\infty C_{sp}}-1\,.$$
	then we absorb the right-hand side with parts of the left-hand side of the inequality involving $\Psi$:
	\begin{align*}   
		\dv{}{t}\Psi
		+\frac{\nu}{8}E_1+\frac{k^2}{8}\beta_0\xi(t) E_4+ A\nu k^2 \gamma_0 c_\infty E_0 &\leq 0\\
	\end{align*}
	Now, using the definition of $\gamma_0$ in \eqref{O:beta_def} and taking $C_\xi$ to be small enough such that $C_{\beta}\leq \frac{1}{128}C_{\infty}C_{\rm{sp}}$ , we then get
	\begin{align}
		\label{O:post_spectral}
		\dv{}{t}\Psi
		+\frac{\nu}{8}E_1+\frac{k^2}{8}\beta_0\xi(t) E_4+C_\xi'(\nu \abs{k})^{1/2}\xi(t)^3 E_0 &\leq 0,
	\end{align}
	where $C_{\beta}'=16 AC_{\beta}^{\frac 32}c_{\infty}$.
	
	It is crucial that this time we do not have a dependency on $\xi(t)$ in $\sigma$ or $A$, therefore we do not obtain a lower bound and hence can allow it to go to zero.
	\\
	From here we can directly move to the closing step.
	\\
	\textbf{Step 3: Rebuilding the functional}\\
	We rewrite \eqref{O:post_spectral} as
	\begin{align*}
		\dv{}{t}\Psi+C_\xi'(\nu \abs{k})^{1/2}\xi(t)^3 
		\qty[E_0+\frac{\nu^{1/2}}{\abs{k}^{1/2}C_\xi \xi(t)^3}E_1+\frac{\abs{k}^{1/2}}{\nu^{1/2}\xi(t) }E_4] &\leq 0\\
		\iff \dv{}{t}\Psi+C_\xi'(\nu \abs{k})^{1/2}\xi(t)^3 
		\qty[E_0+\frac{1}{C_\xi^{3/2} \xi(t)^4}\alpha_0E_1+\frac{k^2}{16 C_\xi^{3/2}\xi(t)^4 }\gamma_0 E_4] &\leq 0.
	\end{align*}
	Since $\xi(t) \leq 1$, we choose a suitable constant $C_{\xi}$ such that $(16 C_{\xi}^{3/2} \xi(t)^4)^{-1}\geq 5$ and hence
	\begin{align}
		\dv{}{t}\Psi+C_\xi'(\nu \abs{k})^{1/2}\xi(t)^3 
		\qty[E_0+5\alpha_0E_1+5k^2\gamma_0 E_4] &\leq 0.
	\end{align}
	From this, the statement of the Lemma follows directly by coercivity of the energy functional.
\end{proof}

\subsection{Illustrative examples of time-modulated flows}\label{intermittent-flow}
In this section, we examine examples which are induced by velocity fields that can be switched on and off at different time intervals. We present two examples: the first involves a velocity field that features a very slow switch on phase and then transitions continuously into a fast polynomial profile. This profile requires a gluing argument to alternate between Theorems \ref{T:on_off} and \ref{T:hyp_weights_main}. The second example features a slowly varying flow that is gradually switched on, held constant for a period, and then switched off. This case is fully addressed by Theorem \ref{T:on_off}, and the decay rate can be directly computed from the upper bound in \eqref{onoff_decay}.

Flows that alternate between the regimes of Theorems \ref{T:hyp_weights_main} and \ref{T:on_off} can be handled by employing a “gluing” approach. Specifically, we will use the end state of one estimate as the initial data for the next. 
Before presenting a concrete example, we will
clarify this procedure in an abstract framework:\\ 
Assume $t_0=0$ and that we have identified times $t_i$, for $i = 1,\dots,N$ such that on each interval $T_i = [t_{i-1}, t_i]$, are either in the regime of Theorem \ref{T:on_off} or Theorem \ref{T:hyp_weights_main}. Our goal here is to determine the dissipation rate in terms of the viscosity without focusing on optimisation of the constants. Therefore, we do not track them explicitly and they can change from line to line. We begin by applying results consecutively in the interval $[t_{N-2},t_{N}]=[t_{N-2},t_{N-1}]\cup [t_{N-1},t_{N}]$, we find that for the final time $t_N$ we have
\begin{align*}
	E_0(t_N) &\leq K_N \,e^{-C_N (\nu\abs{k})^{1/2}\int_{T_N}\xi_N(\tau)\,d\tau} E_0(t_{N-1}) \\
	&\leq K_NK_{N-1} \, e^{-C_N (\nu\abs{k})^{1/2}\int_{T_N}\xi_N(\tau)\,d\tau} e^{-C_{N-1}(\nu\abs{k})^{1/2}\int_{T_{N-1}}\xi_{N-1}(\tau)\,d\tau}E_0(t_{N-2}) \\
	&= K_NK_{N-1} \, e^{-(C_N \int_{T_N}\xi_N(\tau)\,d\tau + C_{N-1}\int_{T_{N-1}}\xi_{N-1}(\tau)\,d\tau)(\nu\abs{k})^{1/2}} E_0(t_{N-2}).
\end{align*}

Here the $K_{i}$ incorporate all respective prefactors such as the logarithmic corrections, $C_i$ are the general constants in the exponential and $\xi_i(t)$ is either the weight adjusted $\xi(t)$ (as in Theorem \ref{T:hyp_weights_main}) and/or some power of it (as in Theorem \ref{T:on_off}). 
The above procedure can be iterated to obtain
\begin{align*}
	E_0(t_N)&\leq \Pi_i K_{i} e^{-C_i  (\nu\abs{k})^{1/2}\int_{T_i}\xi_i(\tau)d\tau}E_0(0)\leq K e^{-C (\nu\abs{k})^{1/2}\sum_i \int_{T_i} \xi_i(\tau)d\tau}E_0(0)\\
	&=K e^{-C  (\nu\abs{k})^{1/2}\int_{0}^{T_N} \xi^*(\tau)d\tau}E_0(0),
\end{align*}
where $K=\max\{K_{i}\}^N$, $C=\min\{C_i\}$ and $\xi^*(t)=\sum_i\xi_i(t)\bm{1}_{T_i}$.\\
By relabelling, we can write this in the form of a generic $t$ as
\begin{align}
	\label{G:abstract}
	E_0(t)\leq K e^{-C (\nu\abs{k})^{1/2}\int_{0}^{t} \xi^*(\tau)d\tau}E_0(0), \quad \forall t \leq T_N.
\end{align}
Lastly, we note that at times when a flow is inactive, the decay is purely diffusive and the decay rate is dictated by the Poincaré inequality applied to Lemma \ref{L:energy}.

\smallskip
In order to simplify the notation, we will assume $\abs{k}=1$ in the following examples.\\ \\
\textbf{Example A}\\  
Consider the flow
\begin{align}
	\xi_A(t)=\xi_1(t) \bm{1}_{T_1}+\xi_2(t) \bm{1}_{T_2}
\end{align}
with
\begin{align*}
	\xi_1(t)&=\nu^{1/2}t,\, T_1=\qty[0,\frac{1}{\nu^{1/2}}]\\
	\xi_2(t)&=(1+\nu^{1/4} (t-\nu^{-1/2}))^4,\, T_2=\qty[\frac{1}{\nu^{1/2}},\frac{1}{\nu^{3/4}}].
\end{align*}
This flow begins at zero, growing within the framework of Theorem~\ref{T:on_off}. Upon reaching \(\xi(t) = 1\) at \(t = \frac{1}{\nu^{1/2}}\), it transitions into a phase of polynomial growth, eventually entering the admissible class described in Theorem~\ref{T:hyp_weights_main}. Consequently, we must bridge the results of these two theorems, necessitating the use of the aforementioned "gluing" procedure.

We first compute the decay within each interval separately.
At the endpoint of the first interval at $t=\frac{1}{\nu^{1/2}}$ we invoke Theorem \ref{T:on_off} and find
\begin{align}
	\label{ex_a}
	E_0\qty(\frac{1}{\nu^{1/2}})\leq K e^{-C\nu^{1/2}\int_0^{\nu^{-1/2}}\xi_1(\tau)^3d\tau}E_0(0)=K e^{-\frac{C}{4}} E_0(0).
\end{align}
At the endpoint of the second interval at $t=\frac{1}{\nu^{3/4}}$ by Theorem \ref{T:hyp_weights_main} we find  that
\begin{align}
	\begin{split}
	E_0\qty(\frac{1}{\nu^{3/4}})&=K e^{-C\nu^{1/2}\int_{\nu^{-1/2}}^{\nu^{-3/4}}\qty(1+\nu^{1/4} (\tau-\nu^{-1/2}))d\tau}E_0(0)\\
	&=K e^{-\frac{C}{2}\qty(\frac{1}{\nu^{3/4}}+\frac{3}{\nu^{1/4}}-\frac{2}{\nu^{1/2}}-2)} E_0\qty(\frac{1}{\nu^{1/2}}).
\end{split}
\end{align}
Hence, by the "gluing" strategy outlined above, we find 
\begin{align}
	\label{glue}
	E_0\qty(\frac{1}{\nu^{3/4}})&\leq K e^{-\frac{C}{4}}e^{-\frac{C}{2}\qty(\frac{1}{\nu^{3/4}}+\frac{3}{\nu^{1/4}}-\frac{2}{\nu^{1/2}}-2)} E_0(0).
\end{align}
For comparison, an autonomous flow $\xi=1$ would have the decay estimate
\begin{align*}
	E_0\qty(\frac{1}{\nu^{3/4}})&\leq K e^{-C\frac{1}{\nu^{1/4}}} E_0(0).
\end{align*}
We find that, while the glued estimate \eqref{glue} includes lower order corrections, the leading order decay at time $t=\frac{1}{\nu^{3/4}}$ is indeed greater than in the autonomous case. However, since the polynomial acceleration happened after the autonomous enhanced dissipation timescale of $t=\frac{1}{\nu^{1/2}}$ was reached, its effect is small. This is because at timescales of order $\mathcal{O}\qty(\frac{1}{\nu^{3/4}})$ both estimates already scale with negative powers of $\nu$ and therefore in both settings almost all energy is already dissipated at these times. As can be seen from \eqref{ex_a}, the enhanced dissipation timescale mirrors the autonomous case up to a factor of $\frac{1}{4}$.\\ \\
\textbf{Example B}\\
We now consider the example given in the introduction.  We have
\begin{align}
	\xi_B(t)=\xi_1(t) \bm{1}_{T_1}+\xi_2(t) \bm{1}_{T_2}+\xi_3(t) \bm{1}_{T_3}
\end{align}
with
\begin{align*}
	\xi_1(t)&=\nu^{1/2}t,\, T_1=\qty[0,\frac{1}{\nu^{1/2}}]\\
	\xi_2(t)&=1,\, T_2=\qty[\frac{1}{\nu^{1/2}},\frac{1}{\nu^{3/4}}]\\
	\xi_3(t)&=\frac{\nu t-1}{\nu^{1/4} - 1} ,\, T_3=\qty[\frac{1}{\nu^{3/4}},\frac{1}{\nu}].
\end{align*}
This flow consists of a turn-on phase $T_1$, a stationary phase $T_2$ and a turn-off phase $T_3$.
We note that all intervals are in the regime of Theorem \ref{T:on_off}, and we therefore can compute the decay without gluing. 
Invoking Theorem \ref{T:on_off} we find
\begin{align*}
	E_0\qty(\frac{1}{\nu})&\leq K e^{-C\nu^{1/2}\int_0^{\nu^{-1}}\xi_B(\tau)^3d\tau} E_0(0)\\
	&=K e^{-C\nu^{1/2}\qty[\int_0^{\nu^{-1/2}}\xi_1(\tau)^3d\tau+\int_{\nu^{-1/2}}^{\nu^{-3/4}}\xi_2(\tau)^3d\tau+\int_{\nu^{-3/4}}^{\nu^{-1}}\xi_3(\tau)^3d\tau]} E_0(0)\\
	&=K e^{-C\qty[\frac{1}{4}+\frac{1-\nu^{1/4}}{\nu^{1/4}}+\frac{1-\nu^{1/4}}{4\nu^{1/2}}]} E_0(0)=K e^{-C\qty[ \frac{1}{4\nu^{1/2}} -\frac{3}{2}]} E_0(0).
\end{align*}
Note that here $K,C$ are directly given in the statement of Theorem \ref{T:on_off}.
For comparison, the dacay estimate for an autonomous flow $\xi(t)=1$ is
\begin{align*}
	E_0\qty(\frac{1}{\nu})&\leq K e^{-C\frac{1}{\nu^{1/2}}} E_0(0).
\end{align*}
We note that, as expected, the leading-order term retains the $\nu$-scaling of the autonomous setting. However, the presence of an additional factor of $\tfrac{1}{4}$, along with lower-order corrections, results in a slightly weaker overall decay rate.

Further, we observe that since the turn-off phase occurs well after the enhanced dissipation timescale is reached, its effect on the overall estimate is minimal. This is since, for $\nu \ll 1$, almost all the energy is dissipated before the turn-off phase begins. 

\begin{figure}[ht]
	\centering
	\includegraphics[width=5.0truein]{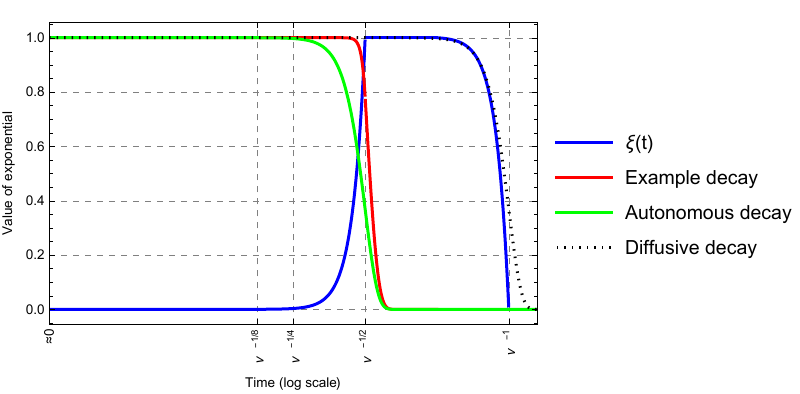}
	\caption{Plot of the exponential decay in Example B, compared to \(\xi(t) = 1\) and pure diffusion as references, all without logarithmic corrections, with \(k = 1\) and \(\nu = 10^{-10}\).
	}
\end{figure}
We conclude this section with some final remarks. In both examples, we observe that the growth restriction in \eqref{class2} affects the decay rate. Since the flow is initially inactive in both cases, we must wait until the autonomous enhanced dissipation timescale, \(t = \frac{1}{\nu^{1/2}}\), for \(\xi(t)\) to reach \(\mathcal{O}(1)\). Consequently, decay is achieved on this timescale. This contrasts with the acceleration discussed in the introduction, where the flow described in \eqref{ex} remains entirely within the admissible class of Theorem \ref{T:hyp_weights_main}.
\appendix
\section{Mixing estimate}
\label{A:mix}
In this section we prove a mixing estimate for solutions of the transport equation on the set $\Omega=\mathbb{T}^2\times [0,\infty)$
\begin{equation}\label{ad}
	\begin{array}{rrl}
		\partial_t \Theta +\bm{u}(x,y,t)\partial_x \Theta&=&0\\
		\Theta(x,y,0)&=&\Theta_0(x,y)\,,
	\end{array}
\end{equation}
where $ \bm{u}(x,y,t)=(\xi(t)v(y), 0)^T$.\\
Taking the horizontal Fourier series in  the $x$-direction, the transport equation reduces to
\begin{equation}\label{eq}
	\partial_t \hat{\Theta} +\xi(t)v(y)i k\hat{\Theta}=0
\end{equation}
with initial data $\hat{\Theta}(k, y, 0)=\hat{\Theta}_0(k,y)$. 

\begin{proposition}
	\label{P:mix}
	Let $ \bm{u}(x,y,t)=(\xi(t)v(y), 0)^T$, where we assume that $\xi(t) > 0 $ and $v(y) \in \mathcal{C}^2(\mathbb{T})$ only admits finitely many simple critical points. Then for $\hat{\Theta}_0(k,y) \in H^1$ the respective solution $\hat{\Theta}(k,y,t)$ to \eqref{ad} satisfies the decay estimate
	\begin{align}
		\label{mix-est}
		\norm{\hat{\Theta}(k,t)}_{H_{y}^{-1}}\leq\min\left\{\frac{C}{(\abs{k}\Xi(t))^{1/2}},1\right\}\norm{\hat{\Theta}_0(k)}_{H_{y}^1},
	\end{align}
	where $\Xi(t) = \int_0^t \xi(\tau) d\tau$ and $C \geq 1$ is independent of $\nu,k$.
\end{proposition}
Here 
\begin{align}
	\label{M:H_minus}
	\norm{\hat{\Theta}(k,t)}_{H_{y}^{-1}}=\sup_{\eta\in H^1:\norm{\eta}_{H_{y}^1}=1} \int_{\mathbb{T}}\hat{\Theta}(k,y,t)\overline{\eta(y)} dy.
\end{align}
denotes the standard $H^{-1}-$norm.
For shear flows, these mixing estimates follow from standard results on oscillatory integrals, usually referred to as the "method of stationary phase". See e.g. Stein \cite[Chapter 8]{Stein}. 
\begin{proof}
	Equation \eqref{eq} can be integrated easily, obtaining
	\begin{align*}
		\hat{\Theta}(k,y,t)=e^{-ik\Xi(t) v(y) }\hat{\Theta}_0(k,y).
	\end{align*}
	We note that, as $\xi(t) > 0$, $\Xi(t)$ grows monotonically in time.\\
	We first look at the region $\Xi(t)\geq 1$:
	by assumption the velocity $v(y)$ only admits a finite number of simple critical points which we label as $y_j$. To simplify the presentation, we may assume that $\partial_y v(y)$ is monotone. Indeed, if this is not the case, we can decompose the domain into a finite number of subintervals over which $\partial_y v(y)$ is monotone. The subsequent argument can then be applied on each subinterval, and summing the resulting estimates establishes the general case, up to a constant that only depends on the number of subintervals. \\Now let $\delta$ be a small parameter to be fixed later. We then define smooth functions $m_j(y)$ such that, for any $j$, $m_j=1$ on the interval $[y_j-\delta,y_j+\delta]$ and $m_j=0$ outside the interval $[y_j-2\delta,y_j+2\delta]$. Additionally, we define $m_0=1-\sum_j m_j$ such that the $m_j$ together with $m_0$ define a partition of unit on $\mathbb{T}$. (See \cite{Bedrossian2017} for an explicit construction.)
	\\
	Now let $\eta(y)$ to be given via \eqref{M:H_minus}. 
	Looking for an estimate on the $H^{-1}$-norm of $\hat{\Theta}$, we need to control the integral
	\begin{align*}
		\int_{\mathbb{T}}\hat{\Theta}(k, y,t)\overline{\eta(y)} dy&=\sum_{j=1} \int_{\mathbb{T}}e^{-ik\Xi(t) v(y) }\phi_j \;dy+ \int_{\mathbb{T}}e^{-ik\Xi(t) v(y) }\phi_0 \;dy\\
		&=\sum_{j=1} \mathcal{I}_j+\mathcal{I}_0,
	\end{align*}
	where $\phi_j(k,y)=\hat{\Theta}_0(k,y)\overline{\eta(y)}m_j(y)$, $\phi_0(k,y)=\hat{\Theta}_0(k,y)\overline{\eta(y)}m_0(y)$. \\
	We treat the integrals around the critical points $\mathcal{I}_j$ and away from the critical points $\mathcal{I}_0$ separately.
	\\
	The integrand in $\mathcal{I}_0$ has support on regions containing no critical points. On this regions, we can integrate by parts, obtaining
	\begin{align}
		\begin{split}
			\label{M:I_0}
			\mathcal{I}_0=\int_{\mathbb{T}}e^{-ik\Xi(t) v(y) }\phi_0 dy&=\int_{\mathbb{T}}-\frac{1}{ik\Xi(t)\partial_y  v(y) }\dv{}{y}\qty(e^{-ik\Xi(t) v(y)  })\phi_0 dy\\
			&=\int_{\mathbb{T}}e^{-ik\Xi(t) v(y) }\dv{}{y}\qty(\frac{1}{ik\Xi(t) \partial_y v(y) }\phi_0 )dy.
		\end{split}	
	\end{align}
	Notice that, given the assumptions on $v$, there exists a constant $C>0$ depending on $v$, but not on $\delta$ and $y$, such that $\abs{\partial_y v(y)}\geq C \delta \quad \forall y \in \text{supp} \,\phi_0 $.  
	Then we estimate the right-hand side in \eqref{M:I_0} as follows:
	\begin{align*}
		\abs{ \mathcal{I}_0}=&\abs{\int_{\mathbb{T}}e^{-ik\Xi(t) v(y) }\dv{}{y}\qty(\frac{1}{ik\Xi(t) \partial_y v(y) }\phi_0 )dy}\leq \frac{1}{\abs{k}\Xi(t)}\int_{\mathbb{T}}\abs{\dv{}{y}\qty(\frac{1}{ \partial_y v(y) }\phi_0)}dy\\
		\leq& \frac{1}{\abs{k}\Xi(t)}\qty[\int_{\mathbb{T}}\abs{\dv{}{y}\frac{1}{ \partial_y v(y) }\phi_0}dy+\int_{\mathbb{T}}\abs{\frac{1}{ \partial_y v(y) } \phi'_0}dy]\\
		\lesssim & \frac{1}{\abs{k}\Xi(t)}\qty[\norm{\phi_0}_{L^\infty_y}\int_{\mathbb{T}\; \cap\; \text{supp}\; \phi_0}\abs{\dv{}{y}\frac{1}{ \partial_y v(y) }}dy +\frac{1}{\delta}\int_{\mathbb{T}\; \cap\; \text{supp}\; \phi_0}\abs{ \phi'_0}dy].
	\end{align*}
	We now recall our initial assumption that $\partial_y v(y)$ is monotone. Then so is $\frac{1}{\partial_y v(y)}$ and hence $\frac{d}{dy}\frac{1}{\partial_y v(y)}$ has a definite sign. This allows us to pull the absolute value out of the first integral, getting
	\begin{align*}
		&\int_{\mathbb{T}\; \cap\; \text{supp}\; \phi_0}\abs{\dv{}{y}\frac{1}{ \partial_y v(y) }}dy=\abs{\int_{\mathbb{T}\; \cap\; \text{supp}\; \phi_0}\dv{}{y}\frac{1}{ \partial_y v(y) }dy}=\abs{\frac{1}{ \partial_y v(y) }\biggr\rvert_{\partial\:\mathbb{T}\; \cap\; \text{supp}\; \phi_0}}\\
		&\leq \frac{1}{C\delta},
	\end{align*}
	where we used again that $\abs{\partial_y v(y)}\geq C \delta$ on the support of $\phi_0$.
	Hence we find 
	\begin{align*}
		\mathcal{I}_0 \lesssim \frac{1}{\abs{k}\Xi(t) \delta} \qty [\norm{\phi_0}_{L^\infty_y}+ \int_{\mathbb{T}\; \cap\; \text{supp}\; \phi_0}\abs{ \phi'_0}dy].
	\end{align*}
	Observing that $\norm{\phi_0}_{L^\infty_y}\leq \norm{\phi_0}_{H^1_y}$ we conclude that
	\begin{align}
		\label{m:I0_final}
		\mathcal{I}_0 \lesssim \frac{1}{\abs{k}\Xi(t) \delta}\norm{\hat{\Theta}_0}_{H^1_y} ,
	\end{align} 
	where we used the definition of $\phi_0$ and $\eta$.
	\\
	\\
	Differently from \eqref{m:I0_final}, the terms $\mathcal{I}_j$ can be easily controlled by bounding the negative exponential
	\begin{align}
		\label{m:IJ_final}
		\mathcal{I}_j=\int_{\mathbb{T}}e^{-ikV(y,t) }\phi_j \;dy \leq 4\delta \norm{\hat\Theta_0(k)}_{H_{y}^1},
	\end{align}
	where we used that $\phi_j$ has support on the regions around the critical points, which is a small interval of length $4\delta$.\\
	We finally conclude 
	\begin{align*}
		\abs{\int_{\mathbb{T}}\hat{\Theta}(k,y,t)\overline{\eta(y)} dy}&=\sum_j \mathcal{I}_j+\mathcal{I}_0\\
		&\lesssim  \qty(4\delta+ \frac{1}{\abs{k}\Xi(t)\delta})\norm{\hat\Theta_0(k)}_{H_y^1}
		\lesssim{\abs{k}^{1/2}\Xi(t)^{1/2}}\norm{\hat\Theta_0(k)}_{H_{y}^1},
	\end{align*}
	where we now applied \eqref{m:I0_final}, \eqref{m:IJ_final} and optimised in $\delta$ by choosing $\delta\sim(|k|\Xi(t))^{-1/2}$.\\
	Taking the supremum in $\eta$, we obtain
	\begin{align}
		\norm{\hat{\Theta}(k,t)}_{H_{y}^{-1}} \leq\frac{C}{\abs{k}^{1/2}\Xi(t)^{1/2}}\norm{\hat\Theta_0(k)}_{H_{y}^1},
	\end{align}
	where $C$ is a universal constant.
	For the region $\Xi(t)\leq 1$, instead, we apply the standard estimate
	$\norm{\hat{\Theta}(k,t)}_{H_{y}^{-1}}\leq \norm{\hat\Theta_0(k)}_{L^2_y}$.
	By combining the two estimates for $\Xi(t)< 1$ and $\Xi(t)\geq 1$, we obtain \eqref{mix-est}.
\end{proof}
We remark that if $\xi(t)$ is bounded from below, we have
$
\Xi(t)\geq \inf_t \xi(t)\: t
$
and hence an estimate of the type
\begin{align*}
	\norm{\hat{\Theta}(k,t)}_{H_{y}^{-1}}\leq\min\left\{\frac{C}{\abs{k}^{1/2}t^{1/2}},1\right\}\norm{\hat{\Theta}_0(k)}_{H_{y}^1}
\end{align*}
holds, where again the constant $C$ might vary. 

For this class of flows, we can directly apply the result of Theorem 2.1 in \cite{Elgindi2019}, yielding an upper bound for solutions of the associated advection–diffusion equation that exhibits enhanced dissipation. However, we note that the resulting dissipation rate is not necessarily optimal.
\section{Energy Balances}

\label{A:lemmas}
\begin{lemma}
	\label{L:energy}
	Let $\theta:\mathbb{T}^2\times(0,\infty)\rightarrow \R$ be a sufficiently smooth solution of \eqref{Heu:adv-eq}. Then, the following identities hold:
	\begin{enumerate}
		\item $\frac{1}{2}\dv{}{t}E_0 = -\nu E_1,$
		\item $\frac{1}{2}\dv{}{t}E_1 = -\nu E_2 - \xi(t)E_3,$
		\item $\dv{}{t}E_3 = -\xi(t)k^2E_4 - 2\nu\Re\inner{ik\partial_y v\partial_y \theta}{\partial_y^2 \theta} - \nu\Re\inner{ik\partial_y^2 v\theta }{\partial_y^2 \theta},$
		\item $\frac{1}{2}\dv{}{t}E_4 = -\nu \norm{\partial_y v\partial_y \theta}_{L^2_y}^2 - 2\nu \Re \inner{\partial_y v\partial_y^2 v \theta}{\partial_y \theta}.$
	\end{enumerate}
\end{lemma}
Since these identities are based on standard energy estimates,  see e.g. \cite{Bedrossian2017,cotizelatigallay}, we only give a quick demonstration. 
\begin{proof}
	In this proof we will repeatedly use the antisymmetric property of the advection term under the $L^2 (\mathbb{T})$ inner product as
	\begin{equation} \label{anitsymmetricprop}
		\Re \inner{ik\xi(t)v(y) f}{f}=0 \quad\forall f\in L^2 (\mathbb{T}).
	\end{equation}
	We also note that boundary terms vanish due to periodicity of the domain. The first identity can be obtained by testing equation \eqref{Heu:adv-eq} with $\theta$ and integrating over $\mathbb{T}$, we have
	\begin{align*}
		\frac{1}{2} \frac{d}{d t} E_0 &= \Re \langle \partial_t \theta, \theta \rangle = \Re \langle \nu \partial_y^2 \theta - i k \xi(t) v \theta, \theta \rangle = - \nu E_1,
	\end{align*}
	where we have used \eqref{anitsymmetricprop}. The second identity follows similarly by testing equation \eqref{Heu:adv-eq} with $-\partial_y^2\theta$, which gives
	\begin{align*}
		\frac{1}{2} \frac{d}{d t} E_1 &= \Re \langle \partial_t \partial_y \theta, \partial_y \theta \rangle = \Re \big\langle \big[ \nu \partial_y^3 \theta - i k \xi(t) v \partial_y \theta - i k \xi(t) \partial_y v \theta \big], \partial_y \theta \big\rangle \\
		&= - \nu E_2 - \xi(t) E_3. 
	\end{align*}
	For the third identity we obtain
	\begin{align*}
		\dv{}{t}E_3 &= \Re\langle ik\partial_y v \partial_t \theta, \partial_y \theta\rangle + \Re\langle ik\partial_y v\theta, \partial_y \partial_t \theta\rangle \\
		&= k^2\xi(t) \Re [\inner{v\partial_y v \theta}{\partial_y \theta}-\inner{\partial_y v \theta}{\partial_y v \theta} -\inner{\partial_y v \theta}{v \partial_y \theta}]\\
		&+\nu \Re [\inner{ik \partial_y v \partial_y^2\theta}{\partial_y \theta}+ \inner{i k \partial_y v \theta}{\partial_y^3 \theta} ] \\
		&= - k^2 \xi(t) E_4 - \nu \Re \langle i k \partial_y^2 v \theta, \partial_y^2 \theta \rangle - 2 \nu \Re \langle i k \partial_y v \partial_y \theta, \partial_y^2 \theta \rangle.
	\end{align*}
	The last identity is then obtained as follows
	\begin{align*}
		\frac{1}{2}\dv{}{t}E_4 &= \Re \langle \partial_y v \partial_t \theta , \partial_y v \theta \rangle = \Re \big\langle \partial_y v \big[ \nu \partial_y^2 \theta - i k \xi(t) v \theta \big], \partial_y v \theta \big\rangle \\
		&= - \nu \Re \inner{\partial_yv\theta}{\partial_y v \partial_y \theta}- 2\nu \Re\inner{\partial_y^2 v \partial_y \theta}{\partial_y v \theta} \\
		&=-\nu \norm{\partial_y v\partial_y \theta}^2-2\nu \Re \inner{\partial_y v\partial_y^2 v \theta}{\partial_y \theta},
	\end{align*}
	where we have integrated the viscous term by parts and used property \eqref{anitsymmetricprop} for the transport term. 
\end{proof}

\section*{Acknowledgments}
The authors extend their gratitude to Daniel Boutros for many insightful discussions. Additionally, C.N. expresses heartfelt thanks to Valerio Nobili, at the Quisisana Clinic Laboratory in Rome, for his demonstrations and detailed explanations of the mixing tools used in his work, which inspired aspects of the content in this study.


\end{document}